\documentclass{article}
\usepackage[left=3.0cm, right=3.0cm]{geometry}
\usepackage[numbers]{natbib}
\usepackage{times}
\usepackage[utf8]{inputenc} % allow utf-8 input
\usepackage[T1]{fontenc}    % use 8-bit T1 fonts
\usepackage{hyperref}       % hyperlinks
\usepackage{url}            % simple URL typesetting
\usepackage{booktabs}       % professional-quality tables
\usepackage{amsfonts}       % blackboard math symbols
\usepackage{nicefrac}       % compact symbols for 1/2, etc.
\usepackage{microtype}      % microtypography
\usepackage{authblk}
\usepackage{enumerate}
\usepackage{array}
\newcolumntype{M}[1]{>{\centering\arraybackslash}m{#1}}
\usepackage[tableposition=top]{caption}
\DeclareCaptionLabelFormat{andtable}{#1~#2  \&  \tablename~\thetable}
\usepackage{subfigure}
\usepackage{abstract}
\usepackage[toc, page]{appendix}

\title{\textbf{A unified construction for series representations and finite approximations of completely random measures}}

\author[1]{Juho Lee\thanks{Corresponding author, \texttt{juho@aitrics.com}}}
\author[2]{Xenia Miscouridou}
\author[2]{Fran\c{c}ois Caron}
\affil[1]{AITRICS, Seoul, South Korea}
\affil[2]{Department of Statistics, University of Oxford, Oxford, United Kingdom}

\usepackage{amssymb, amsmath, amsthm}
\usepackage{mathtools}
\usepackage{bbm}
\usepackage{dsfont}
\usepackage[capitalize,nameinlink]{cleveref}

\newcommand{\calN}{{\mathcal{N}}}

\newcommand{\bbE}{\mathbb{E}}

\newcommand{\bbN}{\mathbb{N}}

\newcommand{\bbR}{\mathbb{R}}

\newcommand{\Real}{{\mathbb R}}

\theoremstyle{plain}
\newtheorem{thm}{Theorem}[section]

\theoremstyle{plain}
\newtheorem{cor}{Corollary}[section]

\theoremstyle{definition}
\newtheorem{defn}{Definition}[section]

\theoremstyle{plain}

\theoremstyle{plain}
\newtheorem{prop}{Proposition}[section]

\theoremstyle{remark}
\newtheorem{rem}{Remark}[section]

\DeclareMathOperator{\gammadist}{Gamma}
\DeclareMathOperator{\igammadist}{iGamma}
\DeclareMathOperator{\betadist}{Beta}

\DeclareMathOperator{\catdist}{Cat}
\DeclareMathOperator{\expdist}{Exp}
\DeclareMathOperator{\unifdist}{Unif}

\DeclareMathOperator*{\crm}{CRM}
\DeclareMathOperator*{\CRM}{CRM}

\DeclareMathOperator*{\bfry}{BFRY}
\DeclareMathOperator*{\etbfry}{etBFRY}
\DeclareMathOperator*{\gbfry}{gBFRY}
\DeclareMathOperator*{\etgbfry}{etgBFRY}
\DeclareMathOperator*{\igbfry}{igBFRY}
\DeclareMathOperator*{\etigbfry}{etigBFRY}

\newcommand{\iidsim}{\overset{\mathrm{i.i.d.}}{\sim}}
\newcommand{\1}[1]{\mathds{1}_{\{#1\}}}
\newcommand{\defas}{\vcentcolon=}  % requires mathtools package

\DeclareMathOperator{\gigdist}{GIG}

\crefname{rem}{Remark}{Remark}

\newcommand{\eqn}[1]{\begin{align}#1 \end{align}}
\newcommand{\eq}[1]{\begin{align*}#1 \end{align*}}

\crefname{appsec}{Appendix}{Appendices}
\crefname{appsubsec}{Appendix}{Appendices}

\usepackage[dvipsnames]{xcolor}
\definecolor{bblue}{rgb}{0,0,0.8}
\hypersetup{
	colorlinks=true,
	citecolor=bblue,
	linkcolor=RedViolet,
	urlcolor=blue
}

\begin{document}

\maketitle

\begin{abstract}
Infinite-activity completely random measures (CRMs) have become important building blocks of complex Bayesian nonparametric models. They have been successfully used in various applications such as clustering, density estimation, latent feature models, survival analysis or network science. Popular infinite-activity CRMs include the (generalized) gamma process and the (stable) beta process. However, except in some specific cases, exact simulation or scalable inference with these models is challenging and finite-dimensional approximations are often considered. In this work, we propose a general and unified framework to derive both series representations and finite-dimensional approximations of CRMs. Our framework can be seen as an extension of constructions based on size-biased sampling of Poisson point process~\citep{Perman1992}. It includes as special cases several known series representations as well as novel ones. In particular, we show that one can get novel series representations for the generalized gamma process and the stable beta process. We also provide some analysis of the truncation error.
\end{abstract}

\section{Introduction}

Infinite-activity completely random measures (CRMs), and more generally functionals of infinite-activity Poisson random measures arise as building blocks of numerous modern structured statistical models. Examples include clustering and density estimation~\citep{Regazzini2003,James2005,James2009}, spatial statistics~\citep{Brix1999,Wolpert1998,Moller2003}, latent factor/trait models~\citep{Griffiths2006,Paisley2009,Zhou2012,Campbell2018,Ayed2019}, network modeling~\citep{Caron2012,Zhou2015,Cai2016,Caron2017}, recommendation systems~\citep{Gopalan2014}, prediction, risk management and option pricing of financial assets~\citep{Cont2004} or survival analysis~\citep{Hjort1990,Nieto-Barajas2004}; see \cite{Lijoi2010} for a review. Popular CRMs include the (generalized) gamma random measure (also known as tempered stable)~\citep{Hougaard1986,Brix1999} or the (stable) beta random measure~\citep{Hjort1990,Teh2009}. Other popular random probability measures such as the Dirichlet process or the Pitman-Yor process are obtained by normalization or transformation of CRMs.\\
The use of statistical models based on infinite-activity CRMs poses a number of practical challenges regarding posterior inference and estimation. Except in some specific cases, most algorithms, either based on Gibbs sampling~\citep{Ishwaran2001,Zhou2015}, slice sampling~\citep{Griffin2011,Foti2012}, mean-field variational inference~\citep{Blei2006,Doshi2009,Lee2016} or sequential Monte Carlo~\citep{Barndorff-Nielsen2001},\citep[Section 3.2.]{Andrieu2010}, require the use of a finite-dimensional approximation of the CRM. Finite-dimensional approximations can either be obtained by (i) truncating a series representation of the CRM, with stochastically decreasing weights, or (ii) by considering a finite measure with $n$ atoms and iid weights, converging in distribution to the CRM as $n$ tends to infinity. For example, for the beta process with scale parameter $\alpha>0$ and probability distribution $H$, the inverse L\'evy series representation is~\citep{Teh2007}
\begin{equation}
G = \sum_{i=1}^\infty W_i\delta_{\theta_i},~~~\text{where}~~W_i=\prod_{j=1}^i \beta_j,~~\beta_i\iidsim\betadist(\alpha,1),~~\theta_i\iidsim H,~~i=1,2\ldots,.\label{eq:betaseries}
\end{equation}
A classical finite-dimensional approximation with iid weights is~\citep{Griffiths2006}
\begin{equation}
\widetilde{G}_n = \sum_{i=1}^n \widetilde W_{n,i}\delta_{\widetilde \theta_{n,i}},~~~\text{where}~~\widetilde W_{n,i}\iidsim\betadist(\alpha/n, 1),~~\widetilde \theta_{n,i}\iidsim H,~~i=1,\ldots,n.\label{eq:betaiid}
\end{equation}
Both representations are routinely used in Markov chain Monte Carlo and variational Bayes approximate inference algorithms~\citep{Doshi2009,Paisley2011,Paisley2012}. Series and iid approximations are similarly used for the gamma process~\citep{Zhou2015,Roychowdhury2015} and the generalized gamma process~\citep{Lee2016}. Since the early work of~\cite{Khintchine1937} and \cite{Ferguson1972} on the so-called inverse L\'evy representation, various generic series representations of Poisson random measures have been proposed~\citep{Bondesson1982,Perman1992,Rosinski1990,Rosinski2001}. Nested series representations have also been recently proposed for some specific CRMs \citep{Paisley2011,Paisley2012,Roychowdhury2015}. Finite iid representations can be obtained using the infinite divisibility properties of the CRM~\citep{Kingman1975} but as noted by~\cite{Lee2016},  it generally does not lead to tractable representations, except in the gamma process case. Other ways of obtaining iid constructions are described in~\citep{Huggins2017} for some family of CRMs. \citep{Campbell2019} provided a recent survey of the existing series representations and approximations as well as a truncation analysis.

The objective of this article is to present a general framework to obtain both series and iid representations of CRMs. Our construction builds on the definition of a Poisson random measure on an extended space; it generalizes the size-biased approach of~\cite{Perman1992}, and admits as special cases both the size-biased and inverse-L\'evy representations. We show that under this construction, one can draw connections between existing series and iid representations that appeared unrelated, and it allows to derive new series and iid representations. More precisely, we show how the iid representation of~\cite{Lee2016} is related to the size-biased construction of~\cite{Perman1992}, derive novel series and iid representations of the generalized gamma and stable beta random measures. We also provide an asymptotic analysis of the truncation error for this class of approximations.\\
This article is organised as follows. In \cref{sec:background} we provide background material on completely random measures and some existing series representations for CRMs and describe the objectives. \cref{sec:generalconstruction} describes the general construction for obtaining series and iid representations of CRMs. In \cref{sec:examples} we describe a number of specific constructions, showing how one recovers some existing constructions as a particular case of our framework. In \cref{sec:asymptotics} we provide an analysis of the asymptotic truncation error, and discuss related approaches in~\cref{sec:discussion}. The proofs and additional background material are provided in the appendix.

\textbf{Notations.} For a measure $\nu$ on $S$ and a positive measurable function $h$ on $S$, write $\nu(h)=\int_S h(x)\nu(dx)$. Let $(\xi_1,\xi_2,\ldots,)$ be the ordered points of a unit-rate Poisson point process on $(0,\infty)$, that is $\xi_1,\xi_2-\xi_1,\xi_3-\xi_2,\ldots$ are iid unit-rate exponential random variables. With a slight abuse of notation, we use the same notation for the distribution of a random variable and its pdf. For instance,
the probability density function (pdf) of a gamma random variable $X \sim \gammadist(a, b)$ is written as $ \gammadist(x ; a, b)$.

%Note to remove: $X$ corresponds to the stochastically ordered points.  $\bar X$ corresponds to the first $n$ stochastically ordered points with a random permutation.

\section{Background}
\label{sec:background}

\subsection{Completely random measures}

%\subsection{General setup}

Let $(S,\mathcal S)$ be a measurable space where $S=(0,\infty)\times \Theta$. For any point $X=(W,\theta)\in S$, we refer to $W>0$ as the \textbf{size} of $X$. Let $N$ be a Poisson random measure on $S$ with mean measure $\nu(dw,d\theta)=\rho(dw)\mu_w(d\theta)$ where $\rho$ is a Borel measure on $(0,\infty)$, called size measure, satisfying
\begin{align}
\int_0^\infty (1-e^{-w})\rho(dw)<\infty\text{ and }\int_0^\infty \rho(dw)=\infty\label{eq:conditionLevy}
\end{align}
 and $\mu_w(\cdot)$ is a Markov probability kernel from $(0,\infty)$ to $\Theta$.
The linear functional $$G=\int_{0}^\infty wN(dw,d\theta)$$ is an infinite-activity completely random measure~\citep{Kingman1967} on $\Theta$ with random weights and random atoms. We write $G\sim \CRM(\rho,\mu_w)$. The conditions \eqref{eq:conditionLevy} imply that the atomic random measures $N$ and $W$ have an infinite number of atoms, and $W(\Theta)$ is almost surely finite. If $\mu_w=H$ does not depend on $w$, the CRM is said to be homogeneous. We assume in the rest of this article that one can easily simulate from $\mu_w$ (or $H$) and/or it admits a tractable density with respect to some reference measure (e.g. Lebesgue). Two popular examples of CRMs are the generalized gamma process (GGP)~\citep{Hougaard1986,Brix1999}, also known as (exponentially) tilted stable process, with size measure
\begin{equation}
\rho(dw)=\frac{\alpha}{\Gamma(1-\sigma)}w^{-1-\sigma}e^{-\tau w}dw\label{eq:LevyGGP}
\end{equation}
where $\alpha>0$, $\sigma\in (0,1)$ and $\tau\geq 0$, or $\sigma\leq 0$ and $\tau> 0$, and the stable beta process (SBP)~\citep{Hjort1990,Teh2009} with
\begin{equation}
\rho(dw) = \frac{\alpha}{B(1-\sigma, c+\sigma)} w^{-\sigma-1} (1-w)^{c+\sigma-1} \1{0<w<1} dw\label{eq:stableSBP},
\end{equation}
where $\sigma\in(-\infty,1)$, $c>-\sigma$, $\alpha>0$ and $B(\cdot,\cdot)$ is the beta function. When $\sigma\geq 0$, both random measures are infinite-activity.
 \medskip

% For any $x\geq 0$, let
%\begin{equation}
%\overline \rho (x)=\int_x^\infty \rho(dw).
%\end{equation}
%We make the following assumptions
%\begin{align}
%\overline \rho (x)&<\infty,~~\text{for all }x>0\\
%\overline \rho (0)&=\infty
%\end{align}
%The first condition ensures that there is almost surely a finite number of points of size greater than $x>0$ while the second conditions implies that $N(S)=\infty$ almost surely.
%  We assume that $\rho$ satisfies $\rho(S)=\infty$, hence $N(S)=\infty$ almost surely.  Additionally, we assume
%\begin{equation}
%\overline \rho (x)=\int_x^\infty \rho(dw)<\infty
%\end{equation}
%for any $x>0$. This condition ensures that there is a.s. a finite number of points of size greater than $x$ for any $x>0$.

%Let $h:S\rightarrow(0,\infty)$ be a measurable, strictly positive function and denote $$W=h(X)>0$$ the \textit{size} of a point $X\in S$. Let $\rho$ be the measure defined as $$\rho(A)=\mu(h^{-1}(A))$$ for any measurable subset $A$ of $(0,\infty)$ and assume that
%\begin{equation}
%\overline \rho (x)=\int_x^\infty \rho(dw)<\infty
%\end{equation}
%for any $x>0$. This condition ensures that there is a.s. a finite number of points of size greater than $x$ for any $x>0$. Denote $\mu_w$ a $\mu$-regular conditional probability law for $X$ given $h(x)=w$.\medskip

\begin{rem}
The constructions described in this paper hold more generally when the first condition in Equation \eqref{eq:conditionLevy} is not satisfied, but $\int_x^\infty \rho(dw)<\infty$ for all $x$. Note that in this case $W(\Theta)=\infty$ almost surely. An example of this more general case is given in \cref{sec:invgammaAT} where $\rho(dw)=w^{-2}dw$.
\end{rem}

\subsection{Objective}

  Our objective is to derive general series representations for the Poisson random measure $N$, or equivalently the CRM $G$, of the form
\begin{equation}
G=\sum_{i=1}^\infty W_i \delta_{\theta_i}
\end{equation}
where the sizes $(W_1,W_2,\ldots)$, are stochastically ordered. That is, for any $w>0$,
$
\Pr(W_{i+1}> w )\leq  \Pr(W_i> w ).
$
We write $W_1 \succeq W_2 \succeq \ldots$ and $X_i=(W_i,\theta_i)$. Denote $G_n$ the measure obtained by truncating the above series after $n$ points
\begin{align}
G_n&=\sum_{i=1}^n W_i\delta_{\theta_i}=\sum_{i=1}^n \overline W_{n,i}\delta_{\overline \theta_{n,i}}\label{eq:truncation}
\end{align}
where $(\overline X_{n,1},\ldots, \overline  X_{n,n})$ is a finitely exchangeable random sequence defined by $\overline X_{n,i}= X_{\pi_n(i)}$ where $\pi_n$ is a random permutation of the set $\{1,\ldots,n\}$. We will refer to the sequence $(X_1,\ldots,X_n)$ (or $(W_1,\ldots,W_n)$)  as the \textbf{sequential truncated representation}, and $(\overline X_{n,1},\ldots,\overline X_{n,n})$ (or $(\overline W_{n,1},\ldots,\overline W_{n,n})$) as the \textbf{exchangeable truncated representation}. In \cref{sec:generalconstruction} we will show that the exchangeable truncated representation can be approximated by a \textbf{finite iid representation}, which will be denoted $(\widetilde X_{n,1},\ldots,\widetilde X_{n,n})$.
%Let $f:S\rightarrow \mathbb R_+$ be some measure function such that $N(h)=\sum_{i\geq 1}f(X_i)<\infty$.

 In the rest of this paper, we will assume that the mean measure $\rho$ is available and one can sample from the conditional distribution $\mu_w$ (or $H$ in the homogeneous case). Under these conditions, we can obtain the representations \eqref{eq:truncation} by first sampling $(W_1,\ldots,W_n)$ (or $(\overline W_{n,1},\ldots,\overline W_{n,n})$), then conditionally sample $(X_1,\ldots,X_n)$ (or $(\overline X_{n,1},\ldots,\overline X_{n,n})$) from $\mu_w$ (or $H$).

\subsection{Existing representations of CRMs}

%In this section we present some existing series representations which we will show are special cases of the general construction presented in \cref{sec:generalconstruction}.
\label{sec:inverselevy}
\label{sec:sizebiased}

\textbf{Inverse-L\'evy representation.} For any $x> 0$, let
\begin{equation}
\overline \rho (x)=\int_x^\infty \rho(dw)
\end{equation}
be the tail intensity of the size measure $\rho$, and denote by
$
\overline\rho^{-1}(y)=\inf\{x>0\mid \overline\rho(x)\leq y \}
$
its generalized inverse. The inverse L\'evy representation~\citep{Khintchine1937,Ferguson1972} is given by
\begin{align*}
W_i=\overline\rho^{-1}(\xi_i).
\end{align*}
In this case, the sizes are ordered $W_1\geq W_2\geq \ldots$ and it therefore leads to the best possible approximation in terms of the sizes. While this representation has been used in many applications~\citep{Wolpert1998a,Nieto-Barajas2004,Griffin2011,Barrios2013,Argiento2016,Arbel2017} its main limitation is that $\rho^{-1}$ is in general non-tractable. Two exceptions are the beta random measure, whose inverse L\'evy representation is given by Equation~\eqref{eq:betaseries}, and the stable random measure (corresponding to the measure \eqref{eq:LevyGGP} with $\sigma\in(0,1)$ and $\tau=0$) where the inverse L\'evy representation is given by
$
W_i=\left ( \frac{\xi_i\sigma\Gamma(1-\sigma)}{\alpha}\right )^{-1/\sigma}.%\frac{\alpha \xi_i^{-\sigma}}{\sigma\Gamma(1-\sigma)}.
$
\vspace{0.1in}

\textbf{Size-biased representation.} The size-biased sequential and exchangeable representations $(W_1,\ldots,W_n)$ and $(\overline W_{n,1},\ldots,\overline W_{n,n})$, introduced by~\cite[Section 4]{Perman1992}, are given as follows\footnote{Note that this is different from what \cite{Campbell2019} call a size-biased representation.}.
Let $0<T_1\leq T_2\leq \ldots$ be defined as $T_i=\Psi^{-1}(\xi_i)$
where $\Psi^{-1}$ is the generalized inverse of the Laplace exponent
$\Psi(t)=\int_0^\infty (1-e^{-wt})\rho(dw)$ and $\mathbb P(W_i\in dw \mid T_i=t)=\frac{w e^{-wt}\rho(dw)}{\psi(t)}$ where
$\psi(t)=\Psi'(t)=\int_0^\infty we^{-wt}\rho(dw).$
Additionally, given $T_{n+1}=t_{n+1}$, we have $\overline W_{n,1},\ldots,\overline W_{n,n}$ are iid with distribution
$$
\mathbb P(\overline W_{n,i}\in dw \mid T_{n+1}=t_{n+1})=\frac{\left (1-e^{-wt_{n+1}}\right )\rho(dw)}{\Psi(t_{n+1})}.
$$
The term size-biased comes from the fact that the atoms are ordered by successively sampling without replacement according to their size $W$
$$
\Pr\left (0<T_1\leq T_2\leq T_3\leq \ldots \mid G \right )= \prod_{k=1}^{\infty}\frac{W_k}{\sum_{j\geq k} W_j}.
$$
In the case of the gamma random measure, which corresponds to Equation~\eqref{eq:LevyGGP} with $\sigma=0$, \cite{Perman1992} show that the series representation corresponds to \cite{Bondesson1982}'s representation and is given by
$
T_i=\Psi^{-1}(\xi_i)=\tau(e^{\xi_i/\alpha} -1)$ and $W_i\mid T_i=t_i\sim \gammadist(1,\tau+t_i)$.

\section{Series representations and finite approximations of CRMs}
\label{sec:generalconstruction}

\subsection{Arrival-time augmentation}
Let $\lambda_w(dt)$ be some Markov probability kernel from $(0,\infty)$ to $(0,\infty)$ with cdf $\Lambda_w(t)=\int_0^t \lambda_w(du)$ satisfying, for any $t>0$
\begin{align}
\int_0^\infty \Lambda_w(t)\rho(dw)<\infty, \quad \Lambda_{w_2}(t)\leq\Lambda_{w_1}(t)\text{ for all }0<w_2\leq w_1.\label{eq:condLambda}
\end{align}
That is, if $T_1\sim\lambda_{w_1}$ and $T_2\sim\lambda_{w_2}$, with $0<w_2\leq w_1$, then $T_2\succeq T_1$.

We consider a Poisson random measure $ N'$ on the augmented space $S'=(0,\infty)\times \Theta\times(0,\infty)$ with mean measure $\nu'(dw,d\theta,dt)=\rho(dw)\mu_w(d\theta)\lambda_w(dt)$. For a point $X'=(W,\theta,T)\in S'$, we refer to $T$ as the \textbf{arrival time} of the point $X'$. Indeed, the second condition in Equation~\eqref{eq:condLambda} ensures that points with larger size $W$ are more likely to have a smaller arrival time $T$. We may therefore consider the following analogy: atoms of the Poisson random measure are enrolled in a race, each atom having a strength $W$, and stronger atoms are more likely to finish faster and therefore have a smaller $T$. The first condition in Equation~\eqref{eq:condLambda} ensures that $N'(\mathbb R_+,\Theta,(0,t))<\infty$ for any $t>0$ hence we can order the arrival times. Let $0<T_1\leq T_2\leq \ldots$ denote the sequence of ordered arrival times, and consider the augmented sequential representation
$
N'=\sum_{i=1}^\infty \delta_{(W_i,\theta_i,T_i)}
$
where $X_i=(W_i,\theta_i)$, $i\geq 1$ are the associated sizes and locations. By the restriction theorem~\citep{Kingman1993}, $N=\sum_{i=1}^\infty \delta_{(W_i,\theta_i)}$ is a Poisson random measure with mean $\rho(dw)\mu_w(d\theta)$ and $W=\int_0^\infty wN(dw,d\theta)\sim\CRM(\rho,\mu_w)$. We now give the general definitions of the sequential, exchangeable and iid representations of the CRM associated to the arrival time kernel $\lambda_w$. For simplicity of presentation, we assume that for any $w$, $\lambda_w$ is absolutely continuous with respect to the Lebesgue measure with $\lambda_w(dt)=\lambda_w(t)dt$, but one can also consider discontinuous cdfs $\Lambda_w$, see \cref{sec:deterministicAT} for an example.

\subsection{Series and truncated exchangeable constructions}

\begin{thm}\label{thm:generalconstruction}
Let $\lambda_w$ be a parametric distribution on $(0,\infty)$ with parameter $w>0$ and $\Lambda_w$ be the associated parametric cumulative density function (cdf) satisfying condition~\eqref{eq:condLambda}.
Consider the conditional distributions
\begin{align}
\phi_t(dw)=\frac{\lambda_w(t)\rho(dw)}{\psi(t)}~~\text{ and }~~\varphi_t(dw) = \frac{\Lambda_w(t)\rho(dw)}{\Psi(t)}
\end{align}
where
$$
\Psi(t)=\int_0^\infty \Lambda_w(t)\rho(dw)~~\text{ and }~~\psi(t)=\int_0^\infty \lambda_w(t)\rho(dw).
$$
The \textbf{sequential construction} $G=\sum_{i=1}^\infty W_i\delta_{\theta_i}$ is obtained as follows, for $i\geq 1$
\begin{align}
T_i=\Psi^{-1}(\xi_i),~~~W_i\mid T_i=t_i \sim \phi_{t_i}~\text{ and }~\theta_i\mid W_i=w_i \sim \mu_{w_i}.\label{eq:seqconst}
\end{align}
The \textbf{truncated exchangeable construction} $G_n=\sum_{i=1}^n \overline W_{n,i}\delta_{\overline \theta_{n,i}}$  is obtained, for $i=1,\ldots,n$ by
\begin{align}
T_{n+1}&=\Psi^{-1}(\xi_{n+1}),~~\overline W_{n,i}\mid T_{n+1}=t_{n+1} \iidsim \varphi_{t_{n+1}}~~\text{ and }~~\overline\theta_{n,i}\mid \overline W_{n,i}=\overline w_{n,i} \sim \mu_{\overline w_{n,i}}.\label{eq:exchangconst}
\end{align}
\end{thm}

\subsection{Finite iid construction}
Note that ${\xi_{n+1}}/{n}$ tends to 1 almost surely as $n$ tends to infinity. This therefore suggests the following \textbf{finite iid construction}, as an approximation to the truncated measure $G_n$
\begin{align}
\widetilde G_n =\sum_{i=1}^n \widetilde W_{n,i} \delta_{\widetilde \theta_{n,i}}~\text{ where }~\widetilde W_{n,i} \iidsim \widetilde \varphi_{n}~~\text{ and }~~&\widetilde\theta_{n,i}\mid \widetilde W_{n,i}=\widetilde w_{n,i} \sim \mu_{\widetilde w_{n,i}},~i=1,\ldots,n,\label{eq:repiid}\\
\shortintertext{where}
\widetilde \varphi_n(dw)=\varphi_{\Psi^{-1}(n)}(dw) &= \frac{\Lambda_w(\Psi^{-1}(n))\rho(dw)}{n}.\label{eq:distiid}
\end{align}

\begin{prop}\label{prop:convergenceiid}
Let $\widetilde G_n$ be the finite iid approximation defined by Equation~\eqref{eq:repiid}. Then $\widetilde G_n$ converges in distribution to $G\sim\CRM(\rho,\mu_w)$ as $n\rightarrow\infty$.
\end{prop}
For the iid construction, one needs to evaluate $\Psi^{-1}(n)$ only once, and this can be done numerically if there is an analytic form for $\Psi$. Instead of the distribution $\widetilde\varphi_n=\varphi_{\Psi^{-1}(n)}$, we can alternatively use more general distributions  $\widetilde\varphi_n=\varphi_{f(n)}$ where $f$ is an increasing function such that $\Psi(f(n))\sim n\text{ as }n\to\infty$.
\cref{prop:convergenceiid} also holds as the proof can be straightforwardly adapted to this case. Note that if $\Psi(t)\sim c t^\sigma$ as $t$ tends to infinity for some constant $c$ and $\sigma>0$, then we can take $f(n)=(n/c)^{1/\sigma}$. \cref{prop:asymptoticPsiinv} gives examples of admissible functions $f$ under generic assumptions on $\rho$ and $\Lambda_w$.

%\end{rem}

\section{Examples}
\label{sec:examples}

We first show how the inverse L\'evy and size-biased constructions described in \cref{sec:inverselevy} can be recovered as special cases of the general construction introduced in \cref{sec:generalconstruction}. We then derive novel constructions within this framework.

\subsection{Deterministic arrival times (inverse-L\'evy construction)}
\label{sec:deterministicAT}

Assume that the arrival times are deterministic given the size $W=w$, and inversely proportional to it, that is
$\lambda_w(dt)=\delta_{1/w}(dt).$
The distribution does not admit a density with respect to the Lebesgue measure, but one can still obtain expressions for the different quantities of interest. We obtain
\begin{align*}
\Lambda_w(t)&=\1{t\geq 1/w},&\Psi(t)&=\overline \rho(1/t),&\Psi^{-1}(\xi)&=1/\overline \rho^{-1}(\xi),\\
\phi_t(dw)&=\delta_{1/t}(dw),&\varphi_t(dw)&= \frac{\rho(dw)\1{w>1/t}}{\overline\rho(1/t)},&\widetilde \varphi_n(dw)&= \frac{\rho(dw)\1{w>\overline\rho^{-1}(n)}}{n}.
\end{align*}
The sequential construction corresponds to the inverse-L\'evy construction described in \cref{sec:inverselevy}. The exchangeable representation is similar to the $\epsilon$-truncation of normalized CRMs, used in~\citep{Argiento2016,Argiento2016a}, except that the truncation threshold $\epsilon=1/T_{n+1}$ is treated as a random variable here.

\subsection{Exponential arrival times (size-biased construction)}
\label{sec:expoAT}

Consider an exponential arrival time distribution with $
\lambda_w(dt)=w e^{-wt}dt$ and $\Lambda_w(t) = 1 - e^{-wt}$. This leads to \cite{Perman1992}'s size-biased sequential and exchangeable representations described in \cref{sec:sizebiased}. While this construction is not novel, it appears that it provides a novel series representation for the generalized gamma random measure. We also show that the iid representation associated to this arrival time distribution corresponds to the finite approximation proposed by \cite{Lee2016}.

\paragraph{Generalized gamma process.} In the case of the size measure \eqref{eq:LevyGGP} with $\alpha > 0, \sigma \in (0,1)$ and $\tau \geq 0$, we obtain the following sequential construction for the GGP, which appears to be novel
\begin{equation}
W_i\mid \xi_i\sim \gammadist\left(1-\sigma,\left( \sigma\xi_i /\alpha + \tau^\sigma\right)^{\frac{1}{\sigma}}\right).
\end{equation}
In~\cref{sec:supp:examples:exponential}, we compare this representation with Rosinski's series representation for the GGP~\citep{Rosinski2001a,Rosinski2007}. 
The conditional distribution for the exchangeable and iid constructions is given by
\begin{equation}
\varphi_t(dw) = \frac{\sigma w^{-1-\sigma}e^{-\tau w} (1-e^{-t w})}{\Gamma(1-\sigma) ((t+\tau)^\sigma - \tau^\sigma))}.\label{eq:BFRY}
\end{equation}
The random variable having this density is called the exponentially-tilted BFRY distribution~\citep{Bertoin2006, Devroye2014, Lee2016},
and written as $\mathrm{etBFRY}(\sigma, t, \tau)$. One can easily simulate from Equation~\eqref{eq:BFRY}, see \cref{subsec:etbfry}. Note that $\Psi(t)\sim\alpha t^\sigma/\sigma$ as $t\to\infty$ hence we can consider the iid distribution $\widetilde\varphi_n(dw)=\varphi_{(n\sigma/\alpha)^{1/\sigma}}(dw)$. This corresponds precisely to the finite-dimensional approximation introduced by~\cite{Lee2016} for the GGP, which can therefore be seen as a particular case of our approach.

\subsection{Gamma arrival times}
\label{sec:gammaAT}

As a generalization of the exponential arrival times, consider now a gamma arrival distribution
$$
\lambda_w(dt)= \frac{t^{\kappa-1}e^{-\kappa w t}(\kappa w)^\kappa}{\Gamma(\kappa)}, \quad
\Lambda_w(t) = \frac{\gamma(\kappa, \kappa wt)}{\Gamma(\kappa)}.
$$
where $\kappa\geq 1$ is a tuning parameter and $\gamma(\kappa,t)=\int_0^t x^{\kappa-1}e^{-x}dx$ is the lower incomplete gamma function. Since $\bbE[T|w] \to 1/w$ and $\mathrm{Var}(T|w) \to 0$ as $\kappa \to \infty$, $T$ converges in probability hence in distribution to $1/w$, and therefore $\Lambda_w(t)\rightarrow \1{t\geq 1/w}$ as $\kappa$ tends to infinity, which corresponds to the arrival time cdf of the inverse L\'evy representation. Hence, the construction based on the gamma arrival times bridges between the size-biased ($\kappa=1$) and inverse-L\'evy ($\kappa\rightarrow\infty$) constructions.

\textbf{Generalized gamma process.}
Consider the generalized gamma process with size measure \eqref{eq:LevyGGP} and parameters $\alpha > 0$, $\sigma \in (0,1)$ and $\tau \geq 0$. We have
\eq{
\psi(t) = \eta \frac{t^{\kappa-1}}{(t + \tau/\kappa)^{\kappa-\sigma}}, \quad
\Psi(t) = \left \{
\begin{array}{ll}
  \eta\left (\frac{\tau}{\kappa}\right )^\sigma B_{\frac{\kappa t}{\kappa t + \tau}}(\kappa, -\sigma) & \text{if }\tau>0 \\
  \frac{\eta}{\sigma} t^\sigma & \text{if }\tau=0
\end{array}\right .
%\frac{\alpha\tau^\sigma\Gamma(\kappa-\sigma)}{\Gamma(\kappa)\Gamma(1-\sigma)} B_{\frac{\kappa t}{\kappa t + \tau}}(\kappa, -\sigma).
}
%$$
%\psi(t) = \frac{\alpha\kappa^\sigma\Gamma(\kappa-\sigma)}{\Gamma(\kappa)\Gamma(1-\sigma)} \frac{t^{\kappa-1}}{(t + \tau/\kappa)^{\kappa-\sigma}}~\text{ and }~
%\Psi(t) = \left \{
%\begin{array}{ll}
%  \frac{\alpha\tau^\sigma\Gamma(\kappa-\sigma)}{\Gamma(\kappa)\Gamma(1-\sigma)} B_{\frac{\kappa t}{\kappa t + \tau}}(\kappa, -\sigma) & \text{If }\tau>0 \\
%  \frac{\alpha\kappa^\sigma\Gamma(\kappa-\sigma)}{\sigma\Gamma(\kappa)\Gamma(1-\sigma)} t^\sigma & \text{If }\tau=0
%\end{array}\right .
%%\frac{\alpha\tau^\sigma\Gamma(\kappa-\sigma)}{\Gamma(\kappa)\Gamma(1-\sigma)} B_{\frac{\kappa t}{\kappa t + \tau}}(\kappa, -\sigma).
%$$
where $\eta=\frac{\alpha\kappa^\sigma\Gamma(\kappa-\sigma)}{\Gamma(\kappa)\Gamma(1-\sigma)}$. For the sequential and exchangeable constructions, we get
\eq{
\phi_t(dw) = \gammadist(w ; \kappa - \sigma, \kappa t + \tau)dw, \quad
\varphi_t(dw) \propto w^{-\sigma-1} e^{-\tau w}\gamma(\kappa, \kappa t w)dw.
}
For the iid construction, we can use \cref{eq:distiid} and estimate $\Psi^{-1}(n)$ numerically or, using \cref{prop:asymptoticPsiinv} and \cref{tab:kernels}, we can alternatively use
$\widetilde\varphi_n(dw)=\varphi_{(\sigma n/\eta)^{1/\sigma}}(dw)$.
The normalizing constant of $\varphi_t$ (and therefore $\widetilde\varphi_n$) has an analytic expression via standard functions. We call the random variable having distribution $\varphi_t$ a \emph{exponentially-tilted generalized BFRY} random variable, due to the form of the pdf obtained by exponentially tilting the pdf of generalized BFRY. This distribution has a number of remarkable properties that make it amenable for tractable simulation and posterior inference. Refer to \cref{subsec:etgbfry} for a more detailed description.

%It has the same conjugacy properties as the gamma distribution, and is therefore a conjugate prior for Poisson, gamma, etc. likelihoods; additionally, its pdf can be expressed as an infinite discrete mixture of gamma distributions, and one can therefore sample exactly from it, without the use of rejection. Refer to \cref{subsec:etgbfry} for a more detailed description of this distribution.

\subsection{Inverse gamma arrival times}
\label{sec:invgammaAT}

Consider now an inverse gamma arrival distribution $\lambda_w(dt)=\igammadist(t ; \kappa, \kappa/w)dt$
where $\igammadist(t; a, b)$ is the pdf of an inverse gamma random variable and $\kappa \geq 1$ is a tuning parameter. By a similar argument as for the gamma arrival times, we have $\Lambda_w(t)\rightarrow \1{t\geq 1/w}$ as $\kappa\to\infty$ hence it also admits the inverse L\'evy construction as a limiting case. The case $\kappa=1$ is of particular interest, as it leads to a tractable novel representation for the GGP (see \cref{sec:supp:examplesIG}), and provide a novel way of interpreting the classical iid approximation of the beta process.

\textbf{Beta process.} Consider the one-parameter beta process with size measure \eqref{eq:stableSBP} with $\sigma=0$ and $c=1$. The bijective transformation $u=-(\alpha\log(w))^{-1}$ gives the measure $\rho(du) = u^{-2} du$ on $(0,\infty)$. Note that $\rho(du)$ is not a L\'evy measure, but we can nonetheless use our construction as the tail L\'evy intensity is finite. Using the inverse gamma kernel with $\kappa=1$, we obtain $\Psi(t)=t$ and the iid distribution $\widetilde\phi_n(du)=\igammadist(u ; 1, 1/n)dt$. Applying the inverse transformation $\widetilde W_i =e^{-1/(\alpha\widetilde U_i)}$, we obtain $\widetilde W_i\sim \betadist(\alpha/n,1)$,
which corresponds to the classical iid approximation for the beta process, described in Equation~\eqref{eq:betaiid}. The iid construction for the beta process can alternatively be recovered using the arrival time distribution $\Lambda_w(t) = w^{\frac{\alpha}{t}}$ directly with \eqref{eq:stableSBP}, without change of variable.
%
%\begin{align}
%\Psi(t)=\int_0^\infty w^{-2}e^{-1/(wt)}dw=t
%\end{align}
%and for the iid model we have $\widetilde U_i\sim \widetilde\phi_n(du)$ where
%$$
%\widetilde\phi_n(du)=\frac{1}{n}u^{-2}e^{-1/(nu)}du
%$$
%which is the distribution of an inverse gamma random variable with parameter $(1,1/n)$. Setting the inverse transformation $\widetilde W_i =e^{-1/(\alpha\widetilde U_i)}$, we obtain
%$$
%\widetilde W_i\sim \betadist(\alpha/n,1)
%$$
%which corresponds to the classical iid approximation for the beta process, described in the introduction.

%This construction can also be obtained directly with different arrival time distribution. Consider
%\begin{equation}
%\lambda_w(dt) = \frac{-\alpha w^{\frac{\alpha}{t}} \log w}{t^2}dt, \quad \Lambda_w(t) = w^{\frac{\alpha}{t}}.
%\end{equation}
%A sample from this distribution can be obtained as
%\begin{equation}
%t' \sim \expdist(\log w^{-\alpha}), \quad t = 1/t'.
%\end{equation}
%Then we have
%\begin{equation}
%\Psi(t) = t, \quad \psi(t) = 1,
%\end{equation}
%and as a result
%\begin{equation}
%\phi_t(dw) = -\frac{\alpha}{t^2} w^{\frac{\alpha}{t}-1} \log w \1{0<w<1} dw, \quad \varphi_t(dw) = \frac{\alpha}{t} w^{\frac{\alpha}{t}-1} \1{0<w<1}dw
%\end{equation}
%A sample from $\phi_t(w)$ can be obtained by
%\begin{equation}
%v \sim \gammadist(2, 1/t), \quad w = e^{-v/\alpha},
%\end{equation}
%and the exchangeable construction correspond to the iid beta approximation.

\subsection{Generalized Pareto arrival time}

Consider the arrival time distribution $\lambda_w(dt) = \frac{cw}{(tw+1)^{c+1}}dt$ where $c>0$.

\textbf{Stable beta process.} Consider the stable beta process with measure \eqref{eq:stableSBP} with $\sigma>0$. We have $\Psi(t)=\frac{\alpha c}{\sigma} ((t+1)^\sigma - 1)$ and
\begin{align*}
\phi_t(dw)& = \frac{w^{-\sigma} (1-w)^{c+\sigma-1} (tw+1)^{-c-1}\1{0<w<1}dw}{B(1-\sigma, c+\sigma) (t+1)^{\sigma-1}}\\
\varphi_t(dw) &= \frac{\sigma w^{-1-\sigma}(1-w)^{c+\sigma-1}}{  cB(1-\sigma, c+\sigma) ((t+1)^\sigma -1)} \bigg(1 - \frac{1}{(tw+1)^c}\bigg) \1{0<w<1}dw.
\end{align*}
These distributions admit the same conjugacy properties as the beta distribution, and one can sample exactly from these distributions as detailed in \cref{sec:supp:examplesGP}.

\section{Truncation error analysis}
\label{sec:asymptotics}

\subsection{Error on functionals of the CRM}

For a measurable function $f:S\rightarrow\mathbb R_+$ such that $N(f)<\infty$ a.s., the error term associated with the truncation is defined as
\begin{equation}
R_n=N(f)-N_n(f)=\sum_{i>n} f(W_i,\theta_i)\label{eq:error}.
\end{equation}
Taking for example $f(w,\theta)=w$ corresponds to the $L_1$ error between the CRM $W$ and $W_n$.

\begin{prop}\label{prop:errormgf}
For $\xi \sim \gammadist(n+1, 1)$, $R_n$ has the following moment generating function
\[
\bbE [ e^{-\lambda R_n}] = \bbE_\xi\bigg[\exp\bigg( - \int_S(1-e^{-\lambda f(w,\theta)})(1-\Lambda_w(\Psi^{-1}(\xi)))\rho(dw)\mu_w(dw)\bigg)\bigg].
\]
%
%\mathbb E[e^{-\lambda R_n}]&=\mathbb E[\mathbb E[e^{-\lambda R_n}\mid T_{n+1}]]\\
%&=\mathbb E_{T_{n+1}}\left [ \exp\left (-\int_S (1-e^{-\lambda f(w,\theta)})(1-\Lambda_w(T_{n+1}))\rho(dw)\mu_w(d\theta)\right )\right ] \\
%&=\mathbb E_{\gamma_{n+1}}\left [ \exp\left (-\int_S (1-e^{-\lambda f(w,\theta)})(1-\Lambda_w(\Psi^{-1}(\gamma_{n+1}))\rho(dw)\mu_w(d\theta)\right )\right ] \\
%&=\int_0^\infty \exp\left (-\int_S (1-e^{-\lambda f(w,\theta)})(1-\Lambda_w(\Psi^{-1}(\gamma)))\rho(dw)\mu_w(d\theta)\right )\\
%&\qquad\times\gammadist(\gamma;n+1,1)  d\gamma
%\end{align*}
%where $\gammadist(x;a,b)$ denotes the pdf of the gamma distribution with parameters $a$ and $b$, evaluated at $x$.
\end{prop}

%\begin{cor}

We now consider results for the special case $f(w,\theta)=w$. The mean and variance of the truncation error $R_n$ given $T_{n+1}$ are given by
\begin{align*}
\mathbb E[R_n\mid T_{n+1}]=\int_0^\infty  w (1-\Lambda_w(T_{n+1}))\rho(dw),~~
\mathbb V[R_n\mid T_{n+1}]=\int_0^\infty  w^2 (1-\Lambda_w(T_{n+1}))\rho(dw).
\end{align*}
%As $n$ tends to infinity, $T_{n+1}$ tends to infinity almost surely.
%\end{cor}
The next proposition provides an asymptotic expression for the error term, giving insights on how the error relates to the choice of the arrival time distribution $\lambda_w(t)$. The proposition makes some assumptions of regular variation on the mean measure $\rho$. Background on regular variation and Mellin transforms is given in \cref{sec:app:regularvariation} and the proof of \cref{prop:asymptoticerror} is given in \cref{sec:proofasympt}.
\begin{prop}\label{prop:asymptoticerror}
Assume that the mean measure $\rho(dw)$ is absolutely continuous with respect to the Lebesgue measure with density function $\rho(w)$ such that
\begin{align}
\label{eq:RVrho}
\rho(w)\sim \zeta_0 w^{-1-\sigma}~\text{ as }~w\rightarrow 0
\end{align}
where $\sigma\in(0,1)$ and $\zeta_0>0$.  Assume additionally that $\Lambda_{w}(t)=1-k(wt)$ where $k$ is a positive function on
$(0,\infty)$ such that its Mellin transform (see \cref{def:Mellin}) $\check{k}$ converges in some open interval containing $[\sigma-2,\sigma-1]$. Assume additionally that either (i) $k$ is differentiable with derivative $k'$ and that the Mellin transform $\check k'$ of $k'$ is defined in some open interval containing $\sigma-1$, or (ii) that $k(x)=\1{x\leq 1}$. Then we have
\begin{equation}
R_n \sim C_1(\sigma)\zeta_0^{1/\sigma}\sigma^{1-1/\sigma} n^{1-1/\sigma}~~\text{ almost surely as }n\to\infty
\label{eq:almostsureRn}
\end{equation}
where the constant $C_1(\sigma)$ is given by $C_1(\sigma)=(1-\sigma)^{-1}$ if $k(x)=\1{x\leq 1}$ and $C_1(\sigma)=\frac{\check{k}(\sigma-1)}{(-\check k'(\sigma-1))^{1-1/\sigma}}$ if $k$ is differentiable, and only depends on the arrival time distribution $\Lambda_w$ and $\sigma$.
%\begin{align*}
%C_1(\sigma)=
%\left \{
%\begin{array}{ll}
%  \check{k}(\sigma-1) & \text{if } k(x)=\1{x\leq 1}  \\
%  \frac{\check{k}(\sigma-1)}{(-\check k'(\sigma-1))^{1-1/\sigma}} & \text{if $k$ is differentiable}
%\end{array}
%\right .
%\end{align*}
%is a constant depending on the arrival time distribution $\Lambda_w$ and $\sigma$.
\end{prop}
The deterministic, gamma, inverse-gamma (for $\kappa>2-\sigma$) and generalized Pareto (for $c>2-\sigma$) arrival time distributions discussed in \cref{sec:examples} all verify the assumptions of \cref{prop:errormgf}. The associated kernels, Mellin transforms and constants $C_1(\sigma)$ are given in \cref{tab:kernels} in the appendix. Figure \ref{fig:simul}(a) shows the value of the constant $C_1(\sigma)$ for the deterministic and gamma arrival time, with different values of $\kappa$. As indicated in \cref{sec:gammaAT}, the approximation gets closer to the deterministic/inverse L\'evy construction as $\kappa$ increases. Both the GGP and the SBP with $\sigma>0$ verify Equation~\eqref{eq:RVrho}, with $\zeta_0=\frac{\alpha}{\Gamma(1-\sigma)}$ for the GGP and $\zeta_0=\frac{\alpha}{B(1-\sigma,c+\sigma)}$ for the SBP.
We run a simulation study in order to investigate the finite-$n$ properties of the proposed approximations. We report in Figure~\ref{fig:simul}(b-c) the mean and variance of $R_n$ for gamma arrival times, for the stable process and the GGP with $\sigma=0.4$. For the stable process, we also compare to the inverse-L\'evy approximation, as it has an analytic form. As expected, the approximation gets better as the value $\kappa$ increases. Additional simulations for other arrival time distributions are given in~\cref{sec:additional_simulations}.

\begin{figure}
\centering
\subfigure{\includegraphics[width=0.32\linewidth]{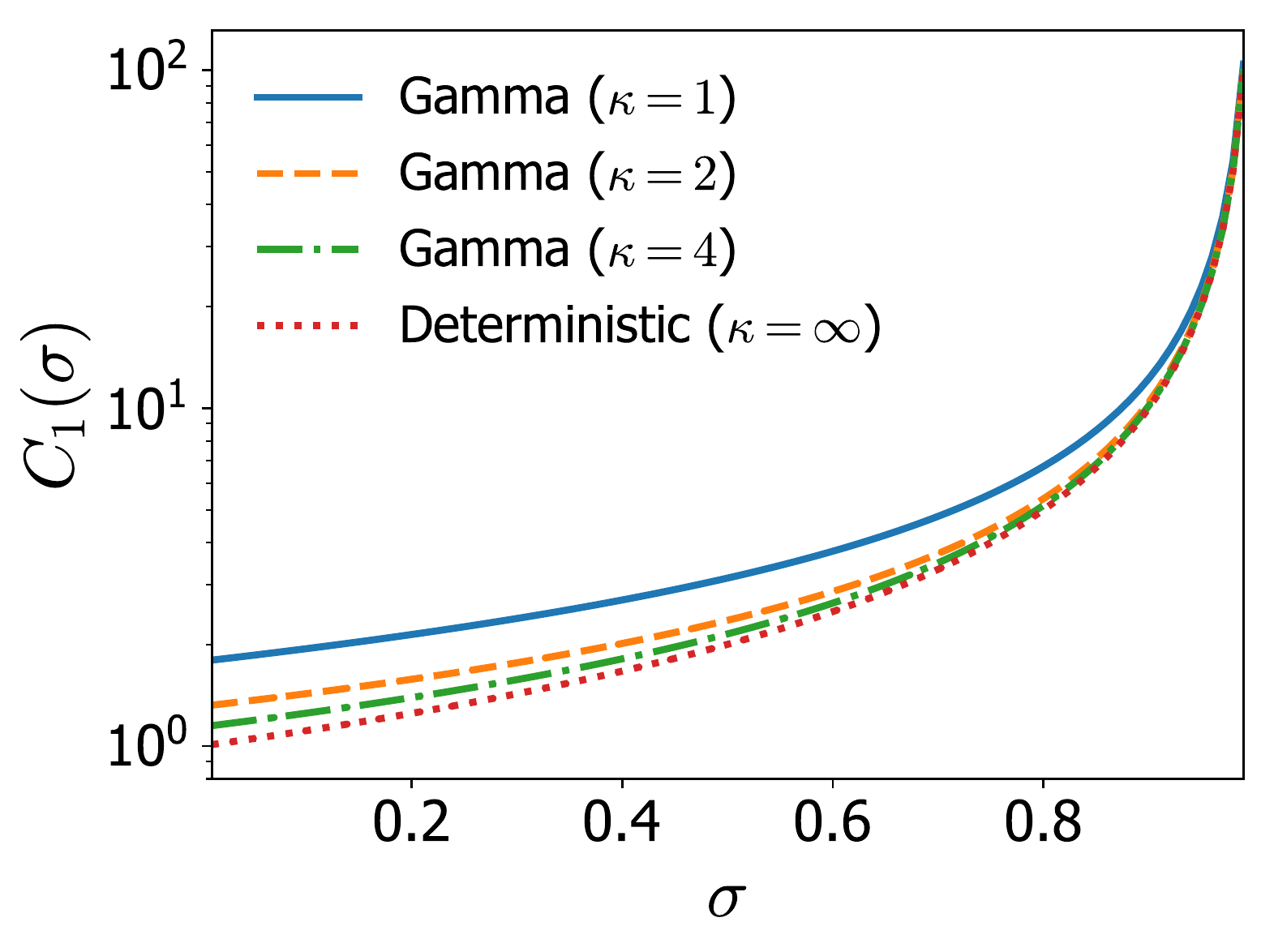}}
\subfigure{\includegraphics[width=0.32\linewidth]{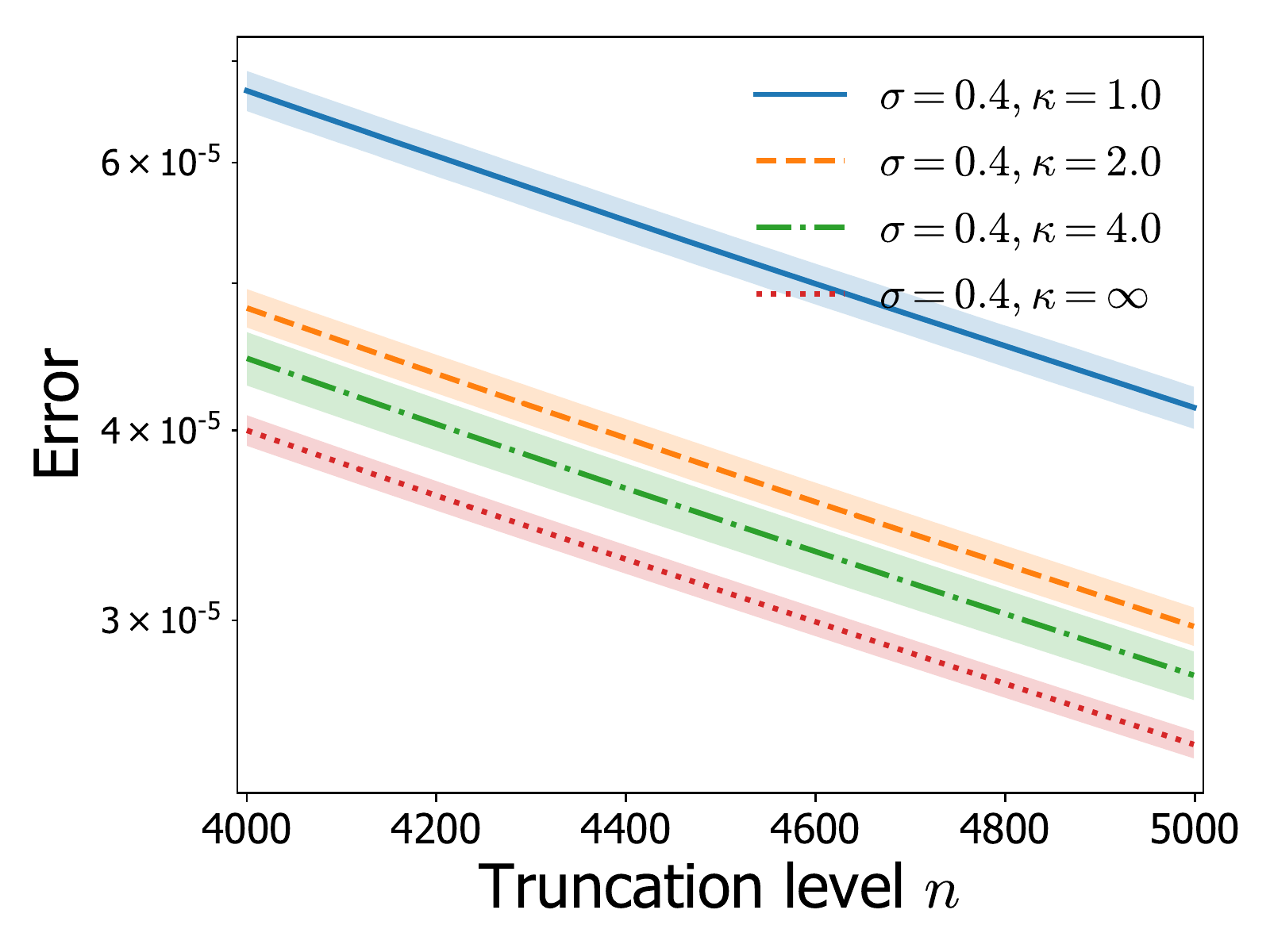}}
\subfigure{\includegraphics[width=0.32\linewidth]{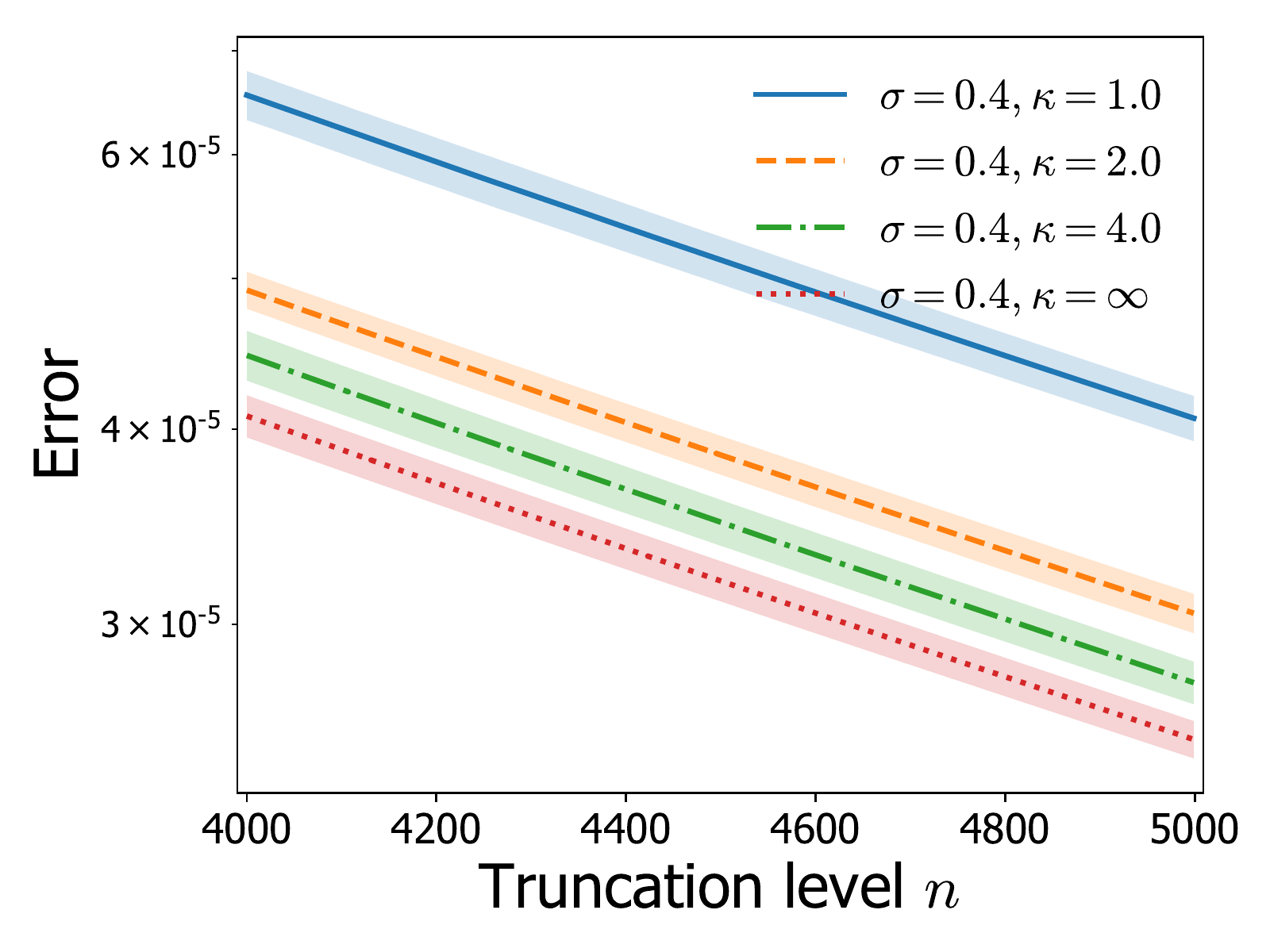}}
%\subfigure[iGamma-SP]{\includegraphics[width=0.32\linewidth]{figures/igamma_stable.pdf}}
\caption{(Left) Constant $C_1(\sigma)$ for deterministic and gamma arrival times;  (Middle-Right) Simulated error $R_n$ with gamma arrival times for (Middle) stable process  and (Right) GGP.}
\label{fig:simul}
\end{figure}

\subsection{$L_1$ error on the marginal likelihood}
In this section we discuss the $L_1$ error on the marginal likelihood when truncated CRMs are used for hierarchical Bayesian models under the framework
described in \cite{Campbell2018}. Let $W \sim \crm(\rho, \mu_w)$, and $W_n$ be its approximation with $n$ atoms. Let $H(\cdot | w)$ be a probability
distribution on $\bbN \cup \{0\}$ for all $w$, and denote $\pi(w) := H(0|w)$. Consider a hierarchical Bayesian model for $m$ observations $X_{1:m} := \{X_j\}_{j=1}^m$,
$
Z_j | w_j \sim H(w_j), \quad X_j | z_j \sim F(z_j), \quad j=1, \dots, m,
$
and denote $p_{m, \infty}(X_{1:m})$ the marginal likelihood for this model. Similarly, denote $p_{m, n}(X_{1:m})$ be the marginal likelihood of the model
with the same generative process, except for $W_n$ instead of $W$. Following \cite{Campbell2018}, we analyze the quality of approximation by comparing $p_{m,\infty}(X_{1:m})$ and $p_{m,n}(X_{1:m})$. For the inverse L\'evy case, one recovers the bound derived in \cite[Theorem D.3.]{Campbell2018}.

\begin{prop}\label{prop:protobound}
We have the bound
$0\leq \frac{1}{2}||p_{m,\infty}(X_{1:m})-p_{m,n}(X_{1:m})||\leq 1-e^{-B_{m,n}}$,
where $B_{m,n}\leq m \int_0^\infty \int_0^\infty  (1- \pi(w))(1-\Lambda_w(\Psi^{-1}(\xi))\rho(dw)   \gammadist(\xi;n+1,1)d\xi$.
\end{prop}

%Similarly to \citep{Campbell2018}, denote
%$$B_{N,n}=\int_0^\infty \int_0^\infty  (1- \pi(w)^N)(1-\Lambda_w(\Psi^{-1}(\gamma))\rho(dw)   \gammadist(\gamma;n+1,1)d\gamma$$ and we can use the further approximation
%$$
%B_{N,n}\leq N \int_0^\infty \int_0^\infty  (1- \pi(w))(1-\Lambda_w(\Psi^{-1}(\gamma))\rho(dw)   \gammadist(\gamma;n+1,1)d\gamma
%$$

\section{Discussion}
\label{sec:discussion}
Our series construction can be seen as a special case of Rosinski's shot-noise series representation~\cite{Rosinski2001} (as is the case for most series constructions, see \cite{Rosinski2001}), using the disintegration $\rho(dw)=\int_0^\infty \nu(\xi,dw) d\xi$ where $\nu(\xi,dw)=\lambda_w(\Psi^{-1}(\xi))\rho(dw)/\psi(\Psi^{-1}(\xi))$ is a Markov kernel (noting that $\int_0^\infty \lambda_w(\Psi^{-1}(\xi))/\psi(\Psi^{-1}(\xi)))d\xi=1$). \citep{Huggins2017} proposed alternative ways of deriving iid approximations for some classes of CRMs. The approach does not rely on a latent Poisson construction and is therefore different from the approach considered here. We emphasize that the finite iid construction is useful for both simulation and hierarchical Bayesian modeling in various contexts. Using \cref{prop:asymptoticPsiinv}, one can approximate infinite-dimensional priors with finite-dimensional iid distributions without any numerical inversion. See \cref{sec:nggp_mixture} where we discuss an example of our construction applied to normalized GGP mixture models.

\newpage
\bibliographystyle{alpha}
\bibliography{paper}

\appendix
\newpage
\begin{appendices}
\section{Background on regular variation and Mellin transforms}
\label{sec:app:regularvariation}

This background material comes from the book of \cite{Bingham1989}.

\subsection{Definitions}
\begin{defn}[Slowly varying function]
A function $\ell: (0,\infty) \to (0, \infty)$ is \emph{slowly varying} at infinity if for all $c > 0$,
\begin{equation}
\frac{\ell(cx)}{\ell(x)} \to 1 \textrm{ as } x \to \infty.
\end{equation}
\end{defn}

\begin{defn}[Regularly varying function]
A function $f: (0,\infty) \to (0,\infty)$ is \emph{regularly varying} at infinity with exponent $\rho \in \bbR$ if $f(x) = x^\rho\ell(x)$ for some slowly varying $\ell$.
A function $f$ is regularly varying at 0 if $f(1/x)$ is regularly varying at infinity, i.e., $f(x) = x^{-\rho}\ell(1/x)$ for some $\rho\in\bbR$ and slowly varying $\ell$.
\end{defn}

\subsection{Basic theorems for regularly varying functions}
Let $U$ be a regularly function with exponent $\rho$ and slowly varying function $\ell$ locally bounded on $(0,\infty)$.

\begin{thm}
[Karamata's theorem] \citep[Propositions 1.5.8 and 1.5.10]{Bingham1989}.
Suppose that $U(t) \sim t^\rho \ell(t)$ as $t \to\infty$.
\begin{itemize}
\item When $\rho > -1$,
$$
\int_{0}^{x} U(t)dt\sim \frac{1}{\rho +1}x^{\rho +1}\ell(x) \textrm{ as } x \to \infty.
$$
\item When $\rho < -1$,
$$
\int_{x}^{\infty}U(t)dt\sim -\frac{1}{\rho +1} x^{\rho+1}\ell(x) \textrm{ as } x \to \infty.
$$
\end{itemize}
\label{th:Karamata}
\end{thm}

\begin{cor}
This also holds when $U$ is regularly varying at 0. When $\rho < -1$ and $U(s) \sim s^\rho \ell(1/s)$ as $s \to 0$,
$$
\int_{x}^\infty U(s)ds \sim \frac{-1}{\rho+1} x^{\rho+1} \ell(1/x) \textrm{ as } x \to 0.
$$
\label{th:Karamata2}
\end{cor}

\subsection{Generalized Abelian theorem}

\begin{defn}\label{def:Mellin}
Given a measurable kernel $k:(0,\infty)\rightarrow\infty$ let
\[
\check{k}(z)=\int_{0}^{\infty}t^{-z-1}k(t)dt=\int_0^\infty u^{z-1}k(1/u)du
\]
be its Mellin transform, for $z\in\mathbb{C}$ such that the integral converges.
\end{defn}

\begin{rem}
  If $k(x)$ has a Mellin transform $\check k(z)$ which converges in $(z_1,z_2)$, then $h(x)=k(1/x)$ has a Mellin transform $\check h(z)=k(-z)$ which converges in $(-z_2,-z_1)$.
\end{rem}

%\begin{prop}\fc{useless}
%Let $k$ be a measurable kernel such that its Mellin transform converges in $(z_1,z_2)$ and assume that $k(t)=o(t^{z_0+1})$ as $t$ tends to $0$ and infinity, for some $z_0<z_2-1$. Then the Mellin transform of $k'$ converges in $(\max(z_0,z_1-1),z_2-1)$ and is given by
%$$
%\check{k'}(z)=(z+1)\check{k}(z+1)
%$$
%\end{prop}
%\begin{proof}
%Using integration by parts, the Mellin transform of $k'$ is given by
%\begin{align*}
%\int_{0}^{\infty}t^{-z-1}k'(t)dt=[t^{-z-1}k(t)]_0^\infty +(z+1)\int_0^\infty t^{-z-2}k(t)dt
%\end{align*}
%\end{proof}

\begin{thm}\citep[Theorem 4.1.6 page 201]{Bingham1989}\label{th:generalizedAbel1} Let the
Mellin transform $\check{k}$ of \thinspace$k$ converge at least in the strip
$\sigma\leq\operatorname{Re}(z)\leq\tau$, where $-\infty<\sigma<\tau<\infty$.
Let $\rho\in(\sigma,\tau)$, $\ell$ a slowly varying function, $c\in
\mathbb{R}.$ If $f$ is measurable, $f(x)/x^{\sigma}$ is bounded on every
interval $(0,a]$ and%
\[
f(x)\sim cx^{\rho}\ell(x)\text{ as }x\rightarrow\infty
\]
then
\[
\int_{0}^{\infty}k(x/t)f(t)t^{-1}dt\sim c\check{k}(\rho)x^{\rho}\ell(x)\text{
as }x\rightarrow\infty
\]

\end{thm}

The next result is a trivial corollary of \cref{th:generalizedAbel1}, considering limits as $x$ tends to $0$.

\begin{cor}\label{th:generalizedAbel2}
Let the Mellin transform $\check{k}$ of \thinspace$k$ converge at least in the
strip $\tau_{1}\leq\operatorname{Re}(z)\leq\tau_{2}$, where $-\infty<\tau
_{1}<\tau_{2}<\infty$. Let $\rho\in(\tau_{1},\tau_{2})$, $\ell$ a slowly
varying function, $c\in\mathbb{R}.$ If $f$ is measurable, $f(x)x^{-\tau_{2}}$
is bounded on every interval $[a,\infty)$ and%
\[
f(x)\sim cx^{\rho}\ell(1/x)\text{ as }x\rightarrow0
\]
then
\[
\int_{0}^{\infty}k(x/t)f(t)t^{-1}dt\sim c\check{k}(\rho)x^{\rho}%
\ell(1/x)\text{ as }x\rightarrow0
\]

\end{cor}

\begin{proof}
\begin{align*}
\int_{0}^{\infty}k(x/t)f(t)t^{-1}dt  &  =\int_{0}^{\infty}k(xu)f(1/u)u^{-1}%
du\\
&  =\int_{0}^{\infty}\widetilde{k}(1/(xu))\widetilde{f}(u)u^{-1}du
\end{align*}
where $\widetilde{f}(x)=f(1/x)$, $\widetilde{f}(x)/x^{-\tau_{2}}$ bounded on
every interval $(0,1/a]$ with%
\[
\widetilde{f}(x)=f(1/x)\sim cx^{-\rho}\ell(x)\text{ as }x\rightarrow\infty
\]
and $\widetilde{k}(x)=k(1/x)$ is such that its Mellin transform converges in
the strip $-\tau_{2}\leq\operatorname{Re}(z)\leq-\tau_{1}$. \cref{th:generalizedAbel1}
above therefore gives the result.
\end{proof}

\section{Proofs}

\subsection{Proof of \cref{thm:generalconstruction}}

%\begin{proof}
The proof is an adaption of the proof for the size-biased construction in \citep[Section 4]{Perman1992}.
The mean measure $\nu'$ of the Poisson random measure $N'$ can be expressed as
\begin{align*}
\nu'(dw,d\theta,dt)&=\rho(dw)\mu_w(d\theta)\lambda_w(t)dt\\
&=\psi(t)dt\times \frac{\lambda_w(t)\rho(dw)}{\psi(t)}\times \mu_w(d\theta)\\
&=\psi(t)dt\times\phi_t(dw)\times\mu_w(d\theta)
\end{align*}
This is the mean measure of a marked Poisson point process, where $(T_1,T_2,\ldots)$ are the points of an inhomogeneous Poisson point process with intensity $\psi(t)$, hence admit the representation  Equation~\eqref{eq:seqconst}, and the marks $(W_i,\theta_i)$ have conditional distribution $\phi_{t}(dw) \mu_{w}(d\theta)$ as shown in Equation~\eqref{eq:seqconst}.
Let $(\overline X_{n,1}',\ldots,\overline X_{n,n}')=(X'_{\pi_1},\ldots,X'_{\pi_n})$ where $\pi_n$ is a random permutation of $\{1,\ldots,n\}$, and $\overline X_{n,i}'=(\overline T_{n,i},\overline W_{n,i},\overline \theta_{n,i})$. By properties of the Poisson process on the real line, the random variables $\overline T_{n,1},\ldots,\overline T_{n,n}$ are iid given $T_{n+1}=t_{n+1}$, with pdf $$\frac{\psi(t)\1{0<t<t_{n+1}}}{\Psi(t_{n+1})}.$$ Hence, given $T_{n+1}=t_{n+1}$, the marks $\overline W_{n,i}$ and $\overline \theta_{n,i}$ are also iid, with conditional distribution $\varphi_{t_{n+1}}(d\overline w_{n,i})\mu_{\overline w_{n,i}}(d\overline\theta_{n,i})$ where
\begin{align*}
\varphi_{t_{n+1}}(d\overline w_{n,i})&=\mathbb E\left [\phi_{\overline T_{n,i}}(d\overline w_{n,i})\mid T_{n+1}=t_{n+1}\right ] \\
&=\int_0^{t_{n+1}}\frac{\lambda_w(t)\rho(d\overline w_{n,i})}{\psi(t)}\frac{\psi(t)}{\Psi(t_{n+1})} dt\\
&=\frac{\Lambda_w(t_{n+1})\rho(d\overline w_{n,i})}{\Psi(t_{n+1})}.
\end{align*}
%\end{proof}

\subsection{Proof of \cref{prop:convergenceiid}}

The proof is similar to that of \cite[Section 3.1]{Lee2016}. Let $f:\Theta\rightarrow (0,\infty)$ be a measurable function.
%\xm{ $\widetilde W_{n,i} f(\widetilde \theta_{n,i}$) in the first line and $\widetilde W_{n,1} f(\widetilde \theta_{n,1}$) in the second line ?}
\begin{align*}
\mathbb E\left [e^{-\widetilde G_n(f)}\right ]&=\mathbb E\left [e^{-\sum_{i=1}^n \widetilde W_{n,i} f(\widetilde\theta_{n,i})}\right ]\\
&= \mathbb E\left [e^{- \widetilde W_{n,1} f(\widetilde\theta_{n,1})}\right ]^n\\
&=\left (\int_S e^{-wf(\theta)} \widetilde\varphi_n(dw)\mu_w(d\theta)\right )^n\\
&=\left (1-\int_S (1-e^{-wf(\theta)}) \widetilde\varphi_n(dw)\mu_w(d\theta)\right )^n\\
&=\left (1-\frac{1}{n}\int_S (1-e^{-wf(\theta)}) \Lambda_w(\Psi^{-1}(n))\rho(dw)\mu_w(d\theta)\right )^n
\end{align*}
Note that $ \Lambda_w(\Psi^{-1}(n))\leq 1, \forall n$ and  $\Lambda_w(\Psi^{-1}(n))\rightarrow 1$ as $n$ tends to infinity. By the bounded convergence theorem, we therefore have
$$\int_S (1-e^{-wf(\theta)}) \Lambda_w(\Psi^{-1}(n))\rho(dw)\mu_w(d\theta)\rightarrow \int_S (1-e^{-wf(\theta)}) \rho(dw)\mu_w(d\theta)$$
as $n\rightarrow\infty$. Additionally, for any real sequence $(a_n)_{n\geq 1}$ converging to $a$ we have $(1-{a_n}/{n})^n\rightarrow e^{-a}$ as $n\rightarrow\infty$. We therefore obtain
$$
\mathbb E\left [e^{-\widetilde G_n(f)}\right ] \rightarrow \exp\bigg(-\int_S (1-e^{-wf(\theta)})\rho(dw)\mu_w(d\theta)\bigg),
$$
where the right-handside is equal to the Laplace functional $\mathbb E[e^{-G(f)}]$ of the CRM $G\sim \CRM(\rho,\mu_w)$ by Campbell's theorem~\citep{Kingman1993}.

%\subsection{Proof of \cref{prop:asymptoticPsiinv}}
%
%TBD - move relevant bits from the next sections

\subsection{Proof of \cref{prop:errormgf}}

By the marking theorem for Poisson point processes~\cite[Chapter 5]{Kingman1993}, given $T_{n+1}=t_{n+1}$, the random measure $\sum_{i\mid T_i\geq t_{n+1}}\delta_{X_i}$ is a Poisson random measure with mean measure $(1-\Lambda_w(t_{n+1}))\rho(dw)\mu_w(d\theta)$. The result follows from Campbell's theorem and the fact that $T_{n+1}=\Psi^{-1}(\xi_{n+1})$.

\subsection{Proof of \cref{prop:asymptoticerror}}
\label{sec:proofasympt}

We state a slightly more general version of \cref{prop:asymptoticerror}, where the constant $\zeta_0$ in Equation~\eqref{eq:RVrho} can more generally be any slowly varying function $\ell_0(1/x)$. We then prove this generalized proposition.

\begin{prop}\label{prop:asymptoticerrorgeneral}[Slight generalization of \cref{prop:asymptoticerror}]
Assume that the mean measure $\rho(dw)$ is absolutely continuous with respect to the Lebesgue measure with density function $\rho(w)$ such that
\begin{align}
\label{eq:RVrhogen}
\rho(w)\sim w^{-1-\sigma}\ell_0(1/w)\text{ as }w\rightarrow 0
\end{align}
where $\sigma\in(0,1)$ and $\ell_0$ is a slowly varying function.  Assume additionally that $\Lambda_{w}(t)=1-k(wt)$ where $k$ is a positive function on
$(0,\infty)$ such that its Mellin transform (see \cref{def:Mellin}) $\check{k}$ converges in some open interval containing $[\sigma-2,\sigma-1]$. Assume additionally that either (i) $k$ is differentiable with derivative $k'$ and that the Mellin transform $\check k'$ of $k'$ is defined in some open interval containing $\sigma-1$, or (ii) that $k(x)=\1{x\leq 1}$. Then we have
\begin{equation}
R_n \sim C_1(\sigma) n^{1-1/\sigma}\ell_{**}(n)~~\text{ almost surely as }n\to\infty
\label{eq:almostsureRngen}
\end{equation}
where  $\ell_{**}$ is some slowly varying function that depends on $\ell_0$ and $\sigma$ but not $\lambda_w$ and the constant $C_1(\sigma)$ is given by $C_1(\sigma)=(1-\sigma)^{-1}$ if $k(x)=\1{x\leq 1}$ and $C_1(\sigma)=\frac{\check{k}(\sigma-1)}{(-\check k'(\sigma-1))^{1-1/\sigma}}$ if $k$ is differentiable, and only depends on the arrival time distribution $\Lambda_w$ and $\sigma$.
\end{prop}

In order to \cref{prop:asymptoticerrorgeneral}, we first state the following proposition.
\begin{prop}\label{prop:asymptoticPsiinv}
Assume that
\begin{equation}
\overline \rho(x)\sim x^{-\sigma}\ell(1/x)\text{ as }x\to 0\label{eq:tailLevyasymp}
\end{equation}
where $\sigma\in(0,1)$ and $\ell$ is a slowly varying function. Assume additionally that
\begin{equation}
\Lambda_{w}(t)=1-k(wt)
\end{equation}
 where $k$ is a positive and differentiable function on
$(0,\infty)$, with derivative $k'$. Assume that the Mellin transform $\check k'$ of $k'$ is defined in some open interval containing $\sigma-1$. Then
\begin{align}
\Psi(t)\sim -\check k'(\sigma-1)t^{\sigma}\ell(t)~~~\text{ and }~~~\Psi^{-1}(t)\sim (-\check k'(\sigma-1))^{-1/\sigma}t^{1/\sigma} \ell_*(t)~~\text{ as }t\rightarrow\infty
\end{align}
where $\ell_*$ is another slowly varying function depending on $\ell$ and $\sigma$, and defined in Equation~\eqref{eq:ellforPsi}. If $\ell(x)=c$ is constant, then we simply have $\ell_*(x)=c^{-1/\sigma}$. In particular, this is the case for both the generalized gamma process and the stable beta process, which verify condition \eqref{eq:tailLevyasymp} when $\sigma>0$ with $\ell(x)=\frac{\alpha}{\sigma\Gamma(1-\sigma)}$ for the GGP and $\ell(x)=\frac{\alpha}{\sigma B(1-\sigma,c+\sigma)}$ for the SBP.
\end{prop}
\begin{proof}[Proof of \cref{prop:asymptoticPsiinv}]
The assumptions \eqref{eq:RVrhogen} and the first condition in Equation~\eqref{eq:condLambda} both imply that  $$\lim_{w\rightarrow 0}\Lambda_w(t)\overline\rho(w)=0.$$
Using integration by parts
\begin{align*}
\Psi(t)&=\int_0^\infty \Lambda_w(t)\rho(dw)\\
&=[ \Lambda_w(t)\overline\rho(w)]_{w=0}^\infty + \int_0^\infty \frac{\partial \Lambda_w(t)}{dw} \overline\rho(w)dw\\
&=\int_0^\infty \frac{\partial \Lambda_w(t)}{dw} \overline\rho(w)dw.
\end{align*}

Now assume $\Lambda_w(t)=1-k(wt)$ where $k$ is differentiable on $(0,\infty)$. Then
$$
\frac{\partial \Lambda_w(t)}{dw}=-t k'(wt).
$$
If the Mellin transform $\check k'$ of $k'$ is defined in some open interval containing $\sigma-1$, then \cref{th:generalizedAbel2} implies
\begin{align}
\Psi(t)\sim -\check k'(\sigma-1)t^{\sigma}\ell(t)
\end{align}
as $t$ tends to infinity. In the case $k(x)=\1{x\leq 1}$, $k$ is not differentiable, but we have directly $$\Psi(t)=\overline\rho(1/t)\sim t^\sigma\ell(t).$$

 Now we use inversion formulas for regularly varying function to get the asymptotic regime for $\Psi^{-1}(t)$. Assume $\sigma>0$, then~\citep[Lemma 22]{Gnedin2007} implies
\begin{align}
\Psi^{-1}(t)\sim (-\check k'(\sigma-1))^{-1/\sigma}t^{1/\sigma} \ell_*(t)
\end{align}
as $t\rightarrow\infty$, where $\ell_*$ is a slowly varying function defined by
\begin{equation}
\ell_*(t) =(\ell^{1/\sigma}(t^{1/\sigma}))^{\#}\label{eq:ellforPsi}
\end{equation}
where $\ell^{\#}$ denotes the de Bruijn conjugate of the slowly varying function $\ell$~\citep[Theorem 1.5.13]{Bingham1989}. Note that $\ell_*$ only depends on $\ell$ and $\sigma$, but not the arrival time distribution $\Lambda_w(t)$.

\end{proof}

Assume that the mean measure $\rho(dw)$ is absolutely continuous with respect to the Lebesgue measure with density function $\rho(w)$ verifying
\begin{align}
\rho(w)\sim w^{-1-\sigma}\ell_0(1/w)\text{ as }w\rightarrow 0,
\end{align}
where $\sigma\in[0,1]$ and $\ell_0$ is a slowly varying function. Equation~\eqref{eq:RVrhogen} and \citep[Proposition 1.5.8]{Bingham1989} imply that
\begin{equation}
\overline \rho(x)\sim x^{-\sigma}\ell(1/x)
\end{equation}
as $x$ tends to 0 where $\ell$ is a slowly varying function defined by
$$
\ell(1/x)=\left \{
\begin{array}{ll}
  \ell_0(1/x)/\sigma & \textrm{ if } \sigma>0 \\
  \int_x^\infty u^{-1}\ell_0(1/u)du & \textrm{ if } \sigma=0
\end{array}\right . .
$$

Assume additionally that $\Lambda_{w}(t)=1-k(wt)$ where $k$ is a positive function on
$(0,\infty)$ such that its Mellin transform (see \cref{def:Mellin}) $\check{k}$ converges in some open interval containing $[\sigma-2,\sigma-1]$. Note that
\begin{align*}
\mathbb E[R_n\mid T_{n+1}]=\int_0^\infty  w (1-\Lambda_w(T_{n+1}))\rho(dw),~~
\mathbb V[R_n\mid T_{n+1}]=\int_0^\infty  w^2 (1-\Lambda_w(T_{n+1}))\rho(dw).
\end{align*}
As $T_{n+1}$ tends to infinity almost surely as $n$ tends to infinity, \cref{th:generalizedAbel2} implies
\begin{align}
\mathbb E[R_n\mid T_{n+1}] & \sim \check{k}(\sigma-1)T_{n+1}^{\sigma-1}\ell_0
(T_{n+1})\label{eq:expectation1}\\
\mathbb V[R_n\mid T_{n+1}] & \sim \check{k}(\sigma-2)T_{n+1}^{\sigma-2}%
\ell_0(T_{n+1})\label{eq:variance1}
\end{align}
almost surely as $n$ tends to infinity.

As $T_{n+1}=\Psi^{-1}(\gamma_{n+1})$ where $\gamma_{n+1}\sim n$ almost surely as $n$ tends to infinity, using \cref{prop:asymptoticPsiinv} we obtain
\begin{equation}
T_{n+1}\sim (-\check k'(\sigma-1))^{-1/\sigma}n^{1/\sigma}  \ell_*(n) \label{eq:asymptTn}
\end{equation}
almost surely as $n$ tends to infinity. Combining Equation \eqref{eq:asymptTn} with Equations~\eqref{eq:expectation1} and~\eqref{eq:variance1}, we obtain
%\begin{align*}
%\mathbb E[R_n\mid T_{n+1}] & \sim \check{k}(\sigma-1)T_{n+1}^{\sigma-1}\ell_0
%(T_{n+1})\\
%\mathbb V[R_n\mid T_{n+1}] & \sim \check{k}(\sigma-2)T_{n+1}^{\sigma-2}%
%\ell_0(T_{n+1})
%\end{align*}
%gives
\begin{align}
\mathbb E[R_n\mid T_{n+1}] & \sim \frac{\check{k}(\sigma-1)}{(-\check k'(\sigma-1))^{1-1/\sigma}}n^{1-1/\sigma}\ell_*^{\sigma-1}(n)\ell_0(t^{1/\sigma}\ell_*(n))\label{eq:temp1}\\
\mathbb V[R_n\mid T_{n+1}] & \sim \frac{\check{k}(\sigma-2)}{(-\check k'(\sigma-1))^{1-2/\sigma}}n^{1-2/\sigma}\ell_*^{\sigma-2}(n)\ell_0(t^{1/\sigma}\ell_*(n))\label{eq:temp2}
\end{align}
almost surely as $n$ tends to infinity. Note that if $\ell_0(t)=\zeta_0$ is constant, then all the other slowly varying functions are also constant with
\begin{align}
\ell(x)=\frac{\zeta_0}{\sigma},~~~\ell_*(x)=\ell(x)^{-1/\sigma}=\left ( \frac{\zeta_0}{\sigma}\right )^{-1/\sigma}.
\end{align}
Finally, Equation \eqref{eq:almostsureRngen} follows similarly to the proof of \citep[Proposition 2]{Gnedin2007}. Using Chebyshev's inequality
\begin{align*}
\Pr\left (\left | \frac{R_n}{\mathbb E[R_n\mid T_{n+1}]}-1 \right |>\epsilon \mid T_{n+1}\right)\leq \frac{\mathbb V[R_n\mid T_{n+1}]}{\epsilon^2\mathbb E[R_n\mid T_{n+1}]^2 }
\end{align*}
Take $a_n=n^2$. As $$\frac{\mathbb V[R_n\mid T_{n+1}]}{\mathbb E[R_n\mid T_{n+1}]^2 }\asymp n^{-1},$$
by the Borel-Cantelli lemma, given $T_n$, $$R_{a_n}\sim \mathbb E[R_{a_n}|T_{a_n+1}]$$
almost surely as $n\to\infty$. As $R_n$ is decreasing, we have, for any $m^2\leq n\leq (m+1)^2$
$$
\frac{R_{(m+1)^2}}{\mathbb E[R_{m^2}\mid T_{m^2+1}]} \leq \frac{R_n}{\mathbb E[R_n\mid T_{n+1}]}\leq \frac{R_{m^2}}{\mathbb E[R_{(m+1)^2}\mid T_{(m+1)^2+1}]}
$$
and it follows by sandwiching that $$R_{n}\sim \mathbb E[R_{n}|T_{n+1}]$$ almost surely as $n\to\infty$. Combining this with Equation~\eqref{eq:temp1} gives the final result, with $\ell_{**}$ the slowly varying function defined by
$$
\ell_{**}(t)=\ell_*^{\sigma-1}(t)\ell_0(t^{1/\sigma}\ell_*(t)).
$$

Note that in the case of the GGP, the different slowly varying functions are all constant functions
\begin{align}
\ell_0(x)&=\frac{\alpha}{\Gamma(1-\sigma)}\\
\ell(x)&=\frac{\alpha}{\sigma\Gamma(1-\sigma)}\\
\ell_*(x)&=\ell(x)^{-1/\sigma}=\left ( \frac{\alpha}{\sigma\Gamma(1-\sigma)}\right )^{-1/\sigma}\\
\ell_{**}(x)&=\left ( \frac{\alpha}{\sigma\Gamma(1-\sigma)}\right )^{-1+1/\sigma}\frac{\alpha}{\Gamma(1-\sigma)}=\left ( \frac{\alpha}{\Gamma(1-\sigma)}\right )^{1/\sigma}\sigma^{1-1/\sigma}
\end{align}
For the SBP, we have
\begin{align}
\ell_0(x)&=\frac{\alpha}{B(1-\sigma,c+\sigma)}\\
\ell(x)&=\frac{\alpha}{\sigma B(1-\sigma,c+\sigma)}\\
\ell_*(x)&=\ell(x)^{-1/\sigma}=\left ( \frac{\alpha}{\sigma B(1-\sigma,c+\sigma)} \right )^{-1/\sigma}\\
\ell_{**}(x)&=\left (\frac{\alpha}{B(1-\sigma,c+\sigma)}\right )^{1/\sigma}\sigma^{1-1/\sigma}
\end{align}
\subsection{Proof of \cref{prop:protobound}}

From \citep{Campbell2018}, we have the protobound
\begin{align*}
0\leq \frac{1}{2}||p_{N,\infty}(X_{1:N})-p_{N,n}(X_{1:N})||\leq 1-\mathbb P(\text{supp}(X_{1:N})\subseteq \text{supp}(G_n))
\end{align*}
In our case,
\begin{align*}
&\mathbb P(\text{supp}(X_{1:N})\subseteq \text{supp}(G_n))\\ &=\mathbb E\left [\prod_{i=n+1}^{\infty} \pi(W_i)^N\right ]\\
&=\mathbb E\left [\mathbb E \left [e^{N\sum_{i\geq n+1}\log \pi(W_i)}\mid T_{n+1}\right ]\right ]\\
&=\mathbb E\left [ \exp\left (-\int_0^\infty  (1-e^{N\log \pi(w)})(1-\Lambda_w(T_{n+1}))\rho(dw)   \right )\right ]\\
&=\int_0^\infty \exp\left (-\int_0^\infty  (1- \pi(w)^N)(1-\Lambda_w(\Psi^{-1}(\gamma))\rho(dw)   \right )\gammadist(\gamma;n+1,1)d\gamma\\
&\geq \exp\left (- \int_0^\infty \int_0^\infty  (1- \pi(w)^N)(1-\Lambda_w(\Psi^{-1}(\gamma))\rho(dw)   \gammadist(\gamma;n+1,1)d\gamma\right )
\end{align*}
where the last inequality follows from Jensen's inequality.

%\section{Proofs}

\section{Mellin transforms}

\begin{table}[h]
\caption{Kernels, Mellin transforms and asymptotic constants for different arrival time distributions}
\label{tab:kernels}
\small
\centering
\setlength{\tabcolsep}{2.4pt}
\begin{tabular}{@{}lcccccc@{}}
\toprule
  Name & $\Lambda_w(t)$ & $k(x)$ & $\check{k}(-z)$ & $k'(x)$&$-\check{k'}(-z)$ & $C_1(\sigma)$\\
  \midrule
  Deterministic & $\1{t\geq 1/w}$ & $\1{x\leq 1}$ & $z^{-1}$ &  -- & -- & $(1-\sigma)^{-1}$ \\
  Exponential& $1-e^{-wt}$ & $e^{-x}$ & $\Gamma(z)$ & $-e^{-x}$ & $\Gamma(z)$ & $\Gamma(1-\sigma)^{1/\sigma}$ \\
  Gamma &  $1-\frac{\Gamma(\kappa,\kappa wt)}{\Gamma(\kappa)}$ & $\frac{\Gamma(\kappa,x\kappa)}{\Gamma(\kappa)}$ & $\frac{\Gamma(z+\kappa)}{z \Gamma(\kappa)\kappa^z}$ &  $\frac{-\kappa^\kappa x^{\kappa-1}e^{-\kappa x}}{\Gamma(\kappa)}$ & $\frac{\kappa^{1-z}\Gamma(\kappa+z-1)}{\Gamma(\kappa)}$ & $\frac{(\kappa-\sigma)\Gamma(\kappa-\sigma)^{1/\sigma}}{(1-\sigma)\Gamma(\kappa)^{1/\sigma}}$\\
  Inv. gamma & $\frac{\Gamma(\kappa,\kappa /(wt))}{\Gamma(\kappa)}$ & $\frac{\gamma(\kappa,\kappa /x)}{\Gamma(\kappa)}$ &$\frac{\Gamma(\kappa-z)}{z\Gamma(\kappa)\kappa^{-z}}$ & $-\frac{\kappa^\kappa e^{-\kappa/x}}{x^{\kappa+1}\Gamma(\kappa)}$ &
  $\frac{\kappa^{z-1}\Gamma(-z+\kappa+1)}{\Gamma(\kappa)}$ & $\frac{\Gamma(\kappa+\sigma)^{1/\sigma}}{(1-\sigma)(\kappa+\sigma-1)\Gamma(\kappa)^{1/\sigma}}$\\
  Gen. Pareto& $1-(1+wt)^{-c}$ & $(1+x)^{-c}$ & $B(z,c-z)$ & $-c(1+x)^{-c-1}$ & $cB(z,c+1-z)$ & $\frac{B(1-\sigma,c+\sigma-1)}{(cB(1-\sigma,c+\sigma))^{1-1/\sigma}}$
\end{tabular}
\end{table}

\subsection{Deterministic kernel}

Take $k(t)=\1{t\leq 1}.$ Then
\begin{align*}
\check{k}(z)  =\int_{0}^{\infty}t^{-z-1}k(t)dt=\int_0^1 t^{-z-1}=-z^{-1}
\end{align*}
if $z<0$.

\subsection{Gamma kernel}
Take $k(t)=\frac{\Gamma(\kappa,\kappa t)}{\Gamma(\kappa)}$ for $\kappa\geq
1$.
\begin{align*}
\check{k}(z) &  =\int_0^\infty u^{-z-1}k(u)du\\
&  =\frac{1}{\Gamma(\kappa)}\int_{0}^{\infty}u^{-z-1}\int_{\kappa u}^{\infty}%
v^{\kappa-1}e^{- v}dvdu\\
&  =\frac{1}{\Gamma(\kappa)}\int_{0}^{\infty}\int_{0}^{v/\kappa}u^{-z-1}v^{\kappa
-1}e^{- v}dudv\\
&  =\frac{-\kappa^z}{z\Gamma(\kappa)}\int_{0}^{\infty}v^{\kappa-z-1}e^{-v}dv\\
&  =\frac{-\Gamma(\kappa-z)\kappa^z}{z\Gamma(\kappa)}%
\end{align*}
which converges for $z<0$. Additionally, $k'(t)=-\frac{\kappa^\kappa t^{\kappa-1}e^{-\kappa t}}{\Gamma(\kappa)}$ hence
\begin{align*}
\check{k'}(z) &  =-\int_0^\infty t^{-z-1}k'(t)dt\\
&=-\frac{\kappa^\kappa }{\Gamma(\kappa)}\int_0^\infty t^{\kappa-z-2}e^{-\kappa t}dt\\
&=-\frac{ \Gamma(\kappa-z-1)\kappa^{z+1}}{\Gamma(\kappa)}
\end{align*}
which converges for $z<\kappa-1$.

\subsection{Inverse gamma kernel}

Take $k(t)=\frac{\gamma(\kappa,\kappa /t)}{\Gamma(\kappa)}$. Note that if $\kappa=1$,
$k(t)=(1-e^{-1/t})$. Then%
\begin{align*}
\check{k}(z) &  =\int_{0}^{\infty}t^{-z-1}k(t)dt\\
&=\int_{0}^{\infty}u^{z-1}k(1/u)du\\
&  =\int_{0}^{\infty}u^{z-1}\frac{\gamma(\kappa,\kappa u)}{\Gamma(\kappa)}du\\
&  =\frac{1}{\Gamma(\kappa)}\int_{0}^{\infty}u^{z-1}\int_{0}^{\kappa u}v^{\kappa
-1}e^{-v}dvdu\\
&  =\frac{1}{\Gamma(\kappa)}\int_{0}^{\infty}v^{\kappa-1}e^{-v}\int%
_{v/\kappa}^{\infty}u^{z-1}dudv\\
&  =-\frac{1}{z\Gamma(\kappa)\kappa^z}\int_{0}^{\infty}v^{\kappa+z-1}e^{-v}dv\text{ if
}z<0\\
&  =-\frac{\Gamma(\kappa+z)}{z\Gamma(\kappa)\kappa^{z}}\text{ if }\kappa+z>0
\end{align*}
therefore defined for $z\in(-\kappa,0)$. Note that
$$
\check{k}(z)\sim -1/z
$$
as $\kappa$ tends to infinity, which corresponds to the inverse-L\'evy case.

We have
$$
k'(t)=-\frac{\kappa^\kappa}{t^{\kappa+1}\Gamma(\kappa)}e^{-\kappa/t}
$$
and
\begin{align*}
\check{k'}(z) &  = \int_{0}^{\infty}t^{-z-1}k'(t)dt\\
&=\int_{0}^{\infty}u^{z-1}k'(1/u)du\\
&  =-\frac{\kappa^{\kappa}}{\Gamma(\kappa)}\int_{0}^{\infty}u^{z+\kappa}e^{-\kappa u}du\\
&=-\frac{\kappa^{-z-1}\Gamma(z+\kappa+1)}{\Gamma(\kappa)}
\end{align*}
defined for $z\in(-1-\kappa,\infty)$. Note again that $\check{k'}(z)\rightarrow -1$ as $\kappa\rightarrow\infty$ (inverse L\'evy case).

\subsection{Generalized Pareto kernel}

Take $k(t)=\frac{1}{(t+1)^c}$. Then
\begin{align*}
\check{k}(z) &  =\int_{0}^{\infty}\frac{t^{-z-1}}{(t+1)^c}dt=B(-z,c+z)
\end{align*}
for $z\in(-c,0)$.
We have $k'(t)=-\frac{c}{(t+1)^{c+1}}$ hence
\begin{align*}
\check{k'}(z) &  =-c\int_{0}^{\infty}\frac{t^{-z-1}}{(t+1)^{c+1}}dt =-cB(-z,c+1+z)
\end{align*}
defined for $z\in(-c-1,0)$.

\section{BFRY and related distributions}
\label{sec:bfry}
\subsection{BFRY distribution}
The BFRY distribution, first named in \cite{Devroye2014} after the work of Bertoin, Fujita, Roynette, and Yor~\citep{Bertoin2006},
arises much earlier in various contexts~\citep{Pitman1997, Winkel2005}. Recently it was highlighted in \cite{Lee2016} as a finite-dimensional
approximate distribution for stable, generalized gamma, and special case of stable-beta processes. The density of a BFRY distribution with parameter $\sigma < (0, 1)$
is written as
\[
\bfry(dw;\sigma) = \frac{\sigma w^{-1-\sigma}(1-e^{-w})dw.}{\Gamma(1-\sigma)}
\]
One can easily verify that the distribution can be simulated as a ratio of independent gamma and beta random variables.
\[
G \sim \gammadist(1-\sigma, 1), \quad B \sim \betadist(\sigma, 1),\quad \frac{G}{B} \overset{d}= \bfry(\sigma).
\]

\subsection{Exponentially-tilted BFRY distribution}
\label{subsec:etbfry}
In \cite{Lee2016,James2015}, the exponentially-tilted version of BFRY distribution was discussed. The density of exponentially-tiled random variable with parameters $\sigma \in (0,1)$,
$c > 0$ and $\tau > 0$ is
\[
\etbfry(dw; \sigma, t, \tau) = \frac{\sigma w^{-1-\sigma}e^{-\tau w}(1-e^{-tw})dw}{\Gamma(1-\sigma)((t+\tau)^\sigma - \tau^\sigma)}.
\]
Then it is easy to show that
\[
G \sim \gammadist(1-\sigma, 1), \quad U \sim \unifdist(0,1), \quad G \cdot ((t + \tau)^\sigma(1-U) + \tau^\sigma U)^{-\frac{1}{\sigma}} \overset{d}= \etbfry(\sigma, t, \tau).
\]

\label{subsec:bfry}
\subsection{Generalized BFRY distribution}
The generalized BFRY distribution, first discussed in \cite{Ayed2019a}, is obtained by generalizing the sampling procedure of BFRY distribution. The generalized
BFRY distribution with parameter $\sigma\in(0,1)$ and $\kappa > \sigma$ is obtained as
\[
G \sim \gammadist(\kappa-\sigma, 1), \quad B \sim \betadist(\sigma, 1), \quad \frac{G}{B} \overset{d}= \gbfry(\kappa, \sigma).
\]
By a simple algebra, we obtain the density as
\[
\gbfry(dw;\kappa, \sigma) = \frac{\sigma w^{-1-\sigma} \gamma(\kappa, w)dw}{\Gamma(\kappa-\sigma)},
\]
where $\gamma(\cdot,\cdot)$ is the lower incomplete gamma function.

\label{subsec:gbfry}
\subsection{Exponentially-tilted generalized BFRY distribution}
\label{subsec:etgbfry}
The density of exponentially-tilted generalized BFRY distribution with parameter $\sigma\in(0,1), \kappa>\sigma, t>0$ and $\tau>0$ is
\[
\etgbfry(dw;\kappa, \sigma, t, \tau) = \frac{w^{-\sigma-1} e^{-\tau w} \gamma(\kappa, tw)}{\tau^\sigma \Gamma(\kappa-\sigma) B_{\frac{t}{t+\tau}}(\kappa, -\sigma)},
\]
where $B_x(\cdot,\cdot)$ is the incomplete beta function.  A random variable having this distribution can be simulated by rejection sampling. Alternatively, note that
\eq{
\etgbfry(w;\kappa,\sigma,t,\tau) &\propto w^{-\sigma-1} e^{-\tau w} \gamma(\kappa, tw)  \\
&= \sum_{j=0}^\infty \frac{\Gamma(\kappa)(tw)^{j+1}e^{-tw}}{\Gamma(\kappa+j+1)} w^{-\sigma-1} e^{-\tau w} \\
&= \sum_{j=0}^\infty \frac{\Gamma(\kappa)}{\Gamma(\kappa+j+1)} t^{\kappa+j} w^{\kappa+j-\sigma-1} e^{-(t+\tau)w} \\
&= \sum_{j=0}^\infty \frac{t^{\kappa+j}\Gamma(\kappa)\Gamma(\kappa+j-\sigma)}{(t+\tau)^{\kappa+j-\sigma} \Gamma(\kappa+j+1)} \gammadist(w;\kappa+j-\sigma, t+\tau).
}
which means that the distribution is an infinite mixture of gamma distributions with mixing proportion
\[
\frac{\Gamma(\kappa)(t+\tau)^\sigma}{\tau^\sigma \Gamma(\kappa-\sigma)B_{\frac{t}{t + \tau}}(\kappa, -\sigma)} \times
\bigg( \bigg(\frac{t}{t+\tau}\bigg)^{\kappa+j} \frac{\Gamma(\kappa+j-\sigma)}{\Gamma(\kappa+j+1)}
\bigg)_{j\geq 1}.
\]
Hence, sampling is straightforward as first sampling the component $j$ from above infinite discrete distribution and sampling from corresponding gamma distribution.

The expoentially-tilted GBFRY distribution has a nice property to be a conjugate prior for Poisson, gamma, normal with fixed mean, and Pareto. Let $W \sim \etgbfry(\kappa, \sigma, t, \tau)$.
Then, for Poisson,
\[
 X | W \sim \mathrm{Poisson}(\lambda W) \Rightarrow W | X \sim \etgbfry(\kappa, \sigma-X, t, \tau + \lambda).
\]
For gamma,
\[
 X | W \sim \gammadist(a, W) \Rightarrow W | X \sim \etgbfry(\kappa, \sigma-a, t, \tau+X).
\]
For normal,
\[
X | W \sim \calN(\mu, 1/W) \Rightarrow W | X \sim \etgbfry\bigg( \kappa, \sigma-\frac{1}{2}, t, \tau + \frac{(X-\mu)^2}{2}\bigg).
\]
For Pareto,
\[
X | W \sim \mathrm{Pareto}(x_0, W) \Rightarrow W | X \sim \etgbfry( \kappa, \sigma-1, t, \tau + \log(X/x_0)).
\]

%\jl{just a note for me, will remove later}
%The moments are computed as
%\eq{
%\bbE[W^m] &= \frac{1}{Z}\int_0^\infty w^{m-\sigma-1} e^{-\tau w} \gamma(\kappa, tw) dw \quad (Z := \tau^\sigma \Gamma(\kappa-\sigma) B_{\frac{t}{t+\tau}}(\kappa, -\sigma))\\
%&= \frac{1}{Z} \int_0^\infty \int_0^1 t^\kappa w^{m+\kappa-\sigma-1} e^{-\tau w} v^{\kappa-1} e^{-twv} dwdv \\
%&= \frac{\Gamma(m+\kappa-\sigma)}{Z} \int_0^1 t^\kappa \frac{v^{\kappa-1}}{(tv + \tau)^{m+\kappa-\sigma}} dv \\
%&= \frac{\Gamma(m+\kappa-\sigma)}{t^{m-\sigma} Z} \int_0^1 \frac{v^{\kappa-1}}{(v + \tau/t)^{m+\kappa-\sigma}} dv \\
%&= \frac{\Gamma(m+\kappa-\sigma)}{\tau^{m-\sigma} Z} B_{\frac{t}{t+\tau}}(\kappa, m-\sigma) \\
%&= \frac{\Gamma(m+\kappa-\sigma)}{\tau^m\Gamma(\kappa-\sigma)} \frac{B_{\frac{t}{t+\tau}}(\kappa, m-\sigma) }{B_{\frac{t}{t+\tau}}(\kappa, -\sigma)}.
%}

\subsection{Inverse generalized BFRY}
One can also consider the counterpart of generalized BFRY distribution where gamma is replaced with inverse gamma. We define inverse generalized BFRY distribution, whose pdf is written as
\label{subsec:igbfry}
\eq{
\igbfry(w; \kappa, \sigma) &= \frac{\sigma}{\Gamma(\kappa+\sigma)} w^{-1-\sigma} \Gamma(\kappa, w^{-1}) \\
&=\frac{\sigma}{\Gamma(\kappa+\sigma)} \int_1^\infty w^{-\kappa-\sigma-1} v^{\kappa-1} e^{-v/w} dv \\
&= \int_1^\infty v^{-1} \cdot \frac{(w/v)^{-\kappa-\sigma-1} e^{-v/w} }{\Gamma(\kappa+\sigma)} \cdot \sigma v^{-\sigma-1} \\
&= \int_0^1 u \cdot \igammadist(wu ; \kappa+\sigma, 1) \cdot \betadist(u ; \sigma, 1) du.
}
Hence, one can realize that
\[
G \sim \igammadist(\kappa+\sigma, 1), \quad B \sim \betadist(\sigma, 1), \quad
\frac{G}{B} \overset{d}= \igbfry(\kappa, \sigma).
\]
This distribution corresponds to the truncated exchangeable density $\varphi_t(w)$ of stable process with inverse gamma arrival times.

\subsection{Exponentially-tilted inverse generalized BFRY}
\label{subsec:etigbfry}
Finally, we consider an exponentially tilted version of inverse GBFRY distribution, whose pdf is written as
\eq{
\etigbfry(w; \kappa, \sigma, t, \tau) \propto w^{-1-\sigma} e^{-\tau w} \Gamma(\kappa, (tw)^{-1}).
}
Unfortunately, we don't have an analytic expression for the normalization constant. We can still sample from this distribution via rejection sampling.
This distribution arises as the truncated exchangeable density $\varphi_t(w)$ of generalized gamma process with inverse gamma arrival times.

\section{Detailed derivations of the results in \cref{sec:examples} and additional examples}
\label{sec:supp:examples}
\subsection{Exponential arrival times}
\label{sec:supp:examples:exponential}

\paragraph{Generalized gamma.} In the case of the size measure \eqref{eq:LevyGGP} with $\alpha > 0, \sigma \in (0,1)$ and $\tau \geq 0$, we have
\[
\Psi(t) = \frac{\alpha}{\sigma}\{ (t+\tau)^\sigma - \tau^\sigma \}, \quad
\psi(t) = \frac{\alpha}{(\tau + \sigma)^{1-\sigma}}.
\]
The arrival times are thus generated as
\[
T_i =\Psi^{-1}(\xi_i)= \bigg( \frac{\sigma}{\alpha}\xi_i + \tau^\sigma\bigg)^{\frac{1}{\sigma}} - \tau.
\]

The conditional distribution for the sequential construction is
\[
\phi_t(dw ) = \gammadist( w ; 1-\sigma, t+\tau)dw.
\]
In summary, the sequential construction for the GGP, is given by
\begin{equation}
W_i\mid \xi_i\sim \gammadist\Bigg(1-\sigma,\bigg( \frac{\sigma}{\alpha}\xi_i + \tau^\sigma\bigg)^{\frac{1}{\sigma}}\Bigg).\label{eq:GGPexpo}
\end{equation}

\paragraph{Comparison to Rosinski's series representation for the GGP}
Rosinski~\cite{Rosinski2001a,Rosinski2007} proposed the following series representation for the GGP/tempered stable process
$$
W_i=\min\left (\left (\frac{\xi_i\sigma\Gamma(1-\sigma)}{\alpha}\right )^{-1/\sigma}  ,e_i u_i^{1/\sigma}\right )
$$
where $e_i\overset{iid}{\sim}\expdist(\tau)$,$u_i\overset{iid}{\sim}\unifdist(0,1)$. For $i$ large, $W_i=\left(\frac{\xi_i\sigma\Gamma(1-\sigma)}{\alpha}\right )^{-1/\sigma}$ with high probability, which corresponds to the inverse-L\'evy construction for the stable process, and this construction has the same asymptotic error rate as the inverse-L\'evy construction for the GGP. The asymptotic error of Rosinski's representation is therefore lower than the asymptotic error for the series defined by \cref{eq:GGPexpo}, by a factor $\Gamma(1-\sigma)^{1/\sigma}(1-\sigma)\in(1,2)$, according to~\cref{tab:kernels}.

\subsection{Gamma arrival times}

\textbf{Generalized gamma.} Consider the generalized gamma process with \eqref{eq:LevyGGP}, $\alpha > 0$, $\sigma \in (0,1)$ and $\tau\geq 0$. We have
\eq{
\psi(t) &= \int_0^\infty \frac{\kappa^\kappa w^\kappa t^{\kappa-1} e^{-\kappa w t}}{\Gamma(\kappa)} \cdot \frac{\alpha}{\Gamma(1-\sigma)} w^{-\sigma-1} e^{-\tau w} dw \\
&= \frac{\alpha\kappa^\sigma\Gamma(\kappa-\sigma)}{\Gamma(\kappa)\Gamma(1-\sigma)} \frac{t^{\kappa-1}}{(t + \tau/\kappa)^{\kappa-\sigma}} \\
\Psi(t) &= \int_0^t \psi(s) ds \\
&= \frac{\alpha\kappa^\sigma\Gamma(\kappa-\sigma)}{\Gamma(\kappa)\Gamma(1-\sigma)} \int_0^t \frac{s^{\kappa-1}}{(s + \tau/\kappa)^{\kappa-\sigma}}ds \\
&=\left \{
\begin{array}{ll}
  \frac{\alpha\tau^\sigma\Gamma(\kappa-\sigma)}{\Gamma(\kappa)\Gamma(1-\sigma)} B_{\frac{\kappa t}{\kappa t + \tau}}(\kappa, -\sigma) & \text{if }\tau>0 \\
  \frac{\alpha\kappa^\sigma\Gamma(\kappa-\sigma)}{\sigma\Gamma(\kappa)\Gamma(1-\sigma)} t^\sigma & \text{if }\tau=0
\end{array}\right .
 %\frac{\alpha\tau^\sigma\Gamma(\kappa-\sigma)}{\Gamma(\kappa)\Gamma(1-\sigma)} B_{\frac{\kappa t}{\kappa t + \tau}}(\kappa, -\sigma),
}
where $B_x(a, b)$ is the incomplete beta function. For $\tau=0$, $\Psi^{-1}$ has the analytic expression
$$\Psi^{-1}(\xi)=\bigg( \frac{\sigma\Gamma(\kappa)\Gamma(1-\sigma)\xi}{\alpha \kappa^\sigma \Gamma(\kappa-\sigma)} \bigg)^{\frac{1}{\sigma}}. $$
For $\tau>0$, there is no analytic expression for $\Psi^{-1}$.
For the sequential construction, we get
\[
\phi_t(dw) = \gammadist(w ; \kappa - \sigma, \kappa t + \tau)dw.
\]
For the exchangeable and iid constructions, we obtain
\[
\varphi_t(dw) = \left \{
  \begin{array}{ll}
    \frac{w^{-\sigma-1} e^{-\tau w}\gamma(\kappa, \kappa t w)}{\tau^\sigma \Gamma(\kappa-\sigma) B_{\frac{\kappa t}{\kappa t + \tau}}(\kappa, -\sigma)}dw & \text{if }\tau>0 \\
     \frac{\sigma w^{-1-\sigma}\gamma(\kappa, \kappa tw)}{(\kappa t)^\sigma \Gamma(\kappa-\sigma)}& \text{if }\tau=0 \\
  \end{array}
\right.
\]
When $\tau=0$, $\varphi_t(dw)$ is the distribution of a generalized BFRY distribution (\cref{subsec:gbfry}). When $\tau > 0$, $\varphi_t(dw)$ corresponds to the distribution of exponentially-tilted generalized BFRY (\cref{subsec:etgbfry}).

\subsection{Inverse gamma arrival times}
\label{sec:supp:examplesIG}
Take
\[
\lambda_w(dt) = \igammadist(t ; \kappa, \kappa/w)dt.
\]
\paragraph{Stable process.}
Consider the stable process with size measure \eqref{eq:LevyGGP}, $\alpha > 0$, $\sigma \in (0,1)$ and $\tau = 0$. We have
\[
\psi(t) = \frac{\alpha \Gamma(\kappa + \sigma)}{\kappa^\sigma \Gamma(1-\sigma)\Gamma(\kappa) t^{1-\sigma}},
\quad \Psi(t) = \frac{\alpha\Gamma(\kappa+\sigma) t^\sigma}{\sigma\kappa^\sigma\Gamma(1-\sigma)\Gamma(\kappa)},
\quad
T_i = \bigg(\frac{\sigma \kappa^\sigma \Gamma(\kappa)\Gamma(1-\sigma) \gamma_i}{\alpha \Gamma(\kappa+\sigma)}\bigg)^{\frac{1}{\sigma}}.
\]

For the sequential construction, we have
\[
\phi_t(dw) = \igammadist(dw; \kappa  + \sigma, \kappa/t).
\]
For the exchangeable construction, we get
\[
\varphi_t(dw) = \frac{\sigma\kappa^\sigma}{t^\sigma\Gamma(\kappa+\sigma)} w^{-1-\sigma} \Gamma\bigg(\kappa + \sigma, \frac{\kappa}{tw}\bigg).
\]
which correspond to inverse generalized BFRY distribution. We can sample from this as $x \sim \igammadist(\kappa+\sigma, 1), y\sim \betadist(\sigma, 1), w = \frac{\kappa x}{ty}$. See \cref{subsec:igbfry} for more details.

The case $\kappa=1$ is of particular interest, as it leads to a tractable novel representation for the GGP, and provide a novel way of interpreting the classical iid approximation of the beta process.

\paragraph{Generalized gamma process.}
Consider GGP with size measure \eqref{eq:LevyGGP} with $\alpha > 0, \sigma \in [0, 1)$ and $\tau > 0$. Take inverse gamma arrival time with $\kappa=1$. We have
\begin{equation}
\psi(t) = \frac{2\alpha(\tau t)^{\frac{\sigma+1}{2}} K_{-\sigma-1}(2\sqrt{\tau/t})}{t^2 \Gamma(1-\sigma)},\quad
\Psi(t) = \frac{2\alpha (\tau t)^{\frac{\sigma}{2}}K_{-\sigma}(2\sqrt{\tau/t})}{\Gamma(1-\sigma)},
\end{equation}
where $K_\nu(\cdot)$ is a modified Bessel function of the second kind. Unfortunately, the arrival time is not given analytically, so we may resort to a numerical root finding algorithm
to compute $T_i = \Psi^{-1}(\xi_i)$. The sequential construction is then given by
\begin{equation}
\phi_t(dw) = \frac{(\tau t)^{\frac{-\sigma-1}{2}} w^{-\sigma-2} e^{-\tau w - \frac{1}{wt}}dw}{2K_{-\sigma-1}(2\sqrt{\tau/t})} = \gigdist(dw ; -\sigma-1, 2\tau, 2/t),
\end{equation}
where $\gigdist(p,a,b)$ is a generalized inverse Gaussian distribution with parameters  $p\in\Real, a>0, b>0$. The exchangeable construction is given by
\begin{equation}
\varphi_t(w) = \frac{(\tau t)^{-\frac{\sigma}{2}} w^{-\sigma-1} e^{-\tau w-\frac{1}{wt}}dw}{2K_{-\sigma}(2\sqrt{\tau/t})} = \gigdist(dw;-\sigma, 2\tau, 2/t).
\end{equation}
This particular case, which seems to be novel to the best of our knowledge, is useful because it covers the gamma process ($\sigma = 0$). It also includes the stable process $(\tau = 0)$
as its limiting case - as $\tau \to 0$,
\begin{equation}
\phi_t(dw) \to \igammadist(dw ; \sigma+1, 1/t), \quad \varphi_t(w) \to \igammadist(dw ; \sigma, 1/t).
\end{equation}
The sequential construction is impractical since we have to invert $\Psi$ for each $t_i$, but the exchangeable construction requires only one inversion for $t_{n+1}$.

\paragraph{Beta process.}
Consider the beta process:
\begin{equation}
\rho(dw) = \alpha w^{-1} \1{0<w<1} dw\label{eq:BP},
\end{equation}
Take the bijective transformation $u=-(\alpha\log(w))^{-1}$ which gives the measure on $(0,\infty)$
\begin{equation}
\rho(du) = u^{-2} du
\end{equation}
Note that $\rho(du)$ is not a L\'evy measure. Using the inverse gamma kernel with $\kappa=1$ to obtain a series approximation for $\rho(du)$, we obtain
\begin{align}
\Psi(t)=\int_0^\infty w^{-2}e^{-1/(wt)}dw=t
\end{align}
and for the iid model we have $\widetilde U_i\sim \widetilde\phi_n(du)$ where
$$
\widetilde\phi_n(du)=\frac{1}{n}u^{-2}e^{-1/(nu)}du
$$
which is the distribution of an inverse gamma random variable with parameter $(1,1/n)$. Setting the inverse transformation $\widetilde W_i =e^{-1/(\alpha\widetilde U_i)}$, we obtain
$$
\widetilde W_i\sim \betadist(\alpha/n,1)
$$
which corresponds to the classical iid approximation for the beta process, described in the introduction.

This construction can also be obtained directly with different arrival time distribution. Consider
\begin{equation}
\lambda_w(dt) = \frac{-\alpha w^{\frac{\alpha}{t}} \log w}{t^2}dt, \quad \Lambda_w(t) = w^{\frac{\alpha}{t}}.
\end{equation}
A sample from this distribution can be obtained as
\begin{equation}
t' \sim \expdist(\log w^{-\alpha}), \quad t = 1/t'.
\end{equation}
Then we have
\begin{equation}
\Psi(t) = t, \quad \psi(t) = 1,
\end{equation}
and as a result
\begin{equation}
\phi_t(dw) = -\frac{\alpha}{t^2} w^{\frac{\alpha}{t}-1} \log w \1{0<w<1} dw, \quad \varphi_t(dw) = \frac{\alpha}{t} w^{\frac{\alpha}{t}-1} \1{0<w<1}dw
\end{equation}
A sample from $\phi_t(w)$ can be obtained by
\begin{equation}
v \sim \gammadist(2, 1/t), \quad w = e^{-v/\alpha},
\end{equation}
and the exchangeable construction correspond to the iid beta approximation.

\subsection{Generalized Pareto arrival time distribution}
\label{sec:supp:examplesGP}
Consider the following arrival time distribution,
\begin{equation}
\lambda_w(dt) = \frac{cw}{(tw+1)^{c+1}}dt, \quad \Lambda_w(t) = 1 - \frac{1}{(tw+1)^c}
\end{equation}
where $c>0$.
\paragraph{Stable beta process.} Consider the stable beta process with L\'evy measure
\begin{equation}
\rho(dw) = \frac{\alpha}{B(1-\sigma, c+\sigma)} w^{-1-\sigma} (1-w)^{c+\sigma-1} \1{0<w<1} dw.
\end{equation}
With change of variable $v = \frac{w}{1-w}$, we see that
\begin{align*}
\Psi(t) &= \int_0^1 \frac{\alpha w^{-1-\sigma}(1-w)^{c+\sigma-1}}{B(1-\sigma, c+\sigma)}  \bigg( 1 - \frac{1}{(tw+1)^c} \bigg) dw \\
&= \frac{\alpha}{B(1-\sigma, c+\sigma)}\int_0^\infty  v^{-1-\sigma} \Big( (1+v)^{-c} - ((t+1)v + 1)^{-c} \Big) dv \\
&= \frac{\alpha}{\Gamma(c) B(1-\sigma, c+\sigma)} \int_0^\infty\int_0^\infty v^{-1-\sigma} e^{-vy} (1 - e^{-tyv}) \cdot y^{c-1}e^{-y}  dy dv \\
&= \frac{\alpha \Gamma(1-\sigma)}{\sigma \Gamma(c) B(1-\sigma, c+\sigma)} \int_0^\infty ((t+1)^\sigma - 1) \cdot y^{c + \sigma -1}e^{-y} dydv \\
&= \frac{\alpha c}{\sigma} ((t+1)^\sigma - 1) \\
\psi(t) &= \alpha c (t+1)^{\sigma-1}.
\end{align*}

For the sequential model we obtain
\begin{equation}
\phi_t(dw) = \frac{w^{-\sigma} (1-w)^{c+\sigma-1} (tw+1)^{-c-1}\1{0<w<1}dw}{B(1-\sigma, c+\sigma) (t+1)^{\sigma-1}}.
\end{equation}

With the change of variable $v=\frac{w}{w-1}$, we see that
\begin{align*}
\phi_t(dv) &= \frac{1}{B(1-\sigma, c+\sigma) (t+1)^{\sigma-1}} v^{-\sigma} (tv + v + 1)^{-c-1} dv \\
&= \frac{1}{B(1-\sigma, c+\sigma) (t+1)^{\sigma-1} \Gamma(c+1)} \int_0^\infty v^{-\sigma} y^{c}e^{-tvy-vy-y} dy \\
&=\int_0^\infty y \cdot \frac{(t+1)^{1-\sigma}(vy)^{-\sigma} e^{-(t+1)vy}}{\Gamma(1-\sigma)}
\cdot \frac{y^{c+\sigma-1} e^{-y}}{\Gamma(c+\sigma)}dy,
\end{align*}
and thus a sample from $\phi_t$ can be obtained as
\begin{equation}
z \sim \gammadist(1-\sigma, 1+t), \quad y \sim \gammadist(c+\sigma, 1), \quad w = \frac{z/y}{z/y+1}.
\end{equation}

For the exchangeable model, we have
\begin{equation}
\varphi_t(dw) = \frac{\sigma w^{-1-\sigma}(1-w)^{c+\sigma-1}}{  cB(1-\sigma, c+\sigma) ((t+1)^\sigma -1)} \bigg(1 - \frac{1}{(tw+1)^c}\bigg) \1{0<w<1}dw,
\end{equation}
and by a similar calculation we see that a sample from $\varphi_t$ can be obtained as
\begin{equation}
z \sim \mathrm{etBFRY}(\sigma, t, 1), \quad y \sim \gammadist(c+\sigma, 1), \quad w = \frac{z/y}{z/y+1}.
\end{equation}

%\section{Additional examples}
%\label{sec:additional_examples}
%
%\subsection{Gamma arrival times}
%\fc{remove this later on}
%\paragraph{Stable process.} Consider the stable process with size measure \eqref{eq:LevyGGP}, $\alpha > 0$, $\sigma \in (0,1)$ and $\tau = 0$. We have
%\[
%\psi(t) = \frac{\alpha \kappa^\sigma \Gamma(\kappa-\sigma)}{\Gamma(\kappa)\Gamma(1-\sigma)t^{1-\sigma}}, \quad
%\Psi(t) = \frac{\alpha \kappa^\sigma \Gamma(\kappa-\sigma)t^\sigma}{\sigma\Gamma(\kappa)\Gamma(1-\sigma)}, \quad
%T_i = \bigg( \frac{\sigma\Gamma(\kappa)\Gamma(1-\sigma)\xi_i}{\alpha \kappa^\sigma \Gamma(\kappa-\sigma)} \bigg)^{\frac{1}{\sigma}}
%\]
%For the sequential and exchangeable constructions, we get
%\[
%\phi_t(dw) = \gammadist(dw; \kappa-\sigma, \kappa t)\text{ and }\varphi_t(dw) = \frac{\sigma w^{-1-\sigma}\gamma(\kappa, \kappa tw)}{(\kappa t)^\sigma \Gamma(\kappa-\sigma)},
%\]
%where the right handside is the pdf of generalized BFRY distribution~\citep{Ayed2019a}. A sample from this distribution can be obtained as
%$x \sim \gammadist(\kappa-\sigma, 1)$, $y \sim \betadist(\sigma, 1)$, $w = \frac{x}{\kappa t y}$. See \cref{subsec:gbfry} for more detail on this distribution.

\section{Details on simulations and additional results}
\label{sec:additional_simulations}
We are interested in measuring the error $R_{n}|T_{n+1}$, apparently  not tractable. Hence, we consider the approximate error $R_{n,\hat n}$ defined for $\hat n > n$ as
\[
R_{n, \hat n}  = \sum_{i=n}^{\hat n} W_i,
\]
where we simulate $(W_i)_{i=1}^{\hat n}$ via sequential constructions. When no analytic expression is available (computing $\Psi^{-1}(t)$ for Gamma arrival times for GGP when $\kappa >1$,
computing $\bar\rho^{-1}(\xi)$ for the inverse-L\'evy for GGP), we resort to numerical inversion algorithm. For each configuration, we first sample a series of arrival time sequences,
and conditioned on that sample 100 series of jumps $(W_i)_{i=1}^{\hat n}$ to compute $R_{n, \hat {n}}$. We repeat this procedure 10 times for each arrival time sequence (thus 1,000 jump simulations in total for each configuration) and report the mean and standard deviations. Unless specified otherwise, we use the hyperameters $\alpha=2.0, \tau=1.0$ for Stable Process (SP) and GGP, and set $\hat n = 10^4$.
\cref{fig:simul} in the paper reports $R_{n, \hat n}$ for SP and GGP.

\cref{fig:error_long_range} shows the approximate errors for gamma and inverse gamma arrival times for SP, and gamma arrival times for GGP. We can observe that the variances quickly approach zero, except for the inverse gamma arrival time with $\kappa=1$ case for which our theory predicts to have infinite variance.

\cref{fig:C1} shows the value of constants $C_1(\sigma)$ for gamma and inverse gamma arrival times with varying $\sigma$ and $\kappa$ values. Note that lower $C_1(\sigma)$ implies lower expected error $\bbE[R_n|T_{n+1}]$ by \cref{prop:asymptoticerror}. We found that gamma arrival times exhibit lower $C_1(\sigma)$ when $0 < \sigma < 0.5$, and inverse gamma has lower $C_1(\sigma)$ when $0.5 < \sigma < 1$.
This observation is empirically confirmed in \cref{fig:gamma_vs_igamma}.

Finally, we compared the approximate error to asymptotic value of $R_n$. In case of gamma arrival times for GGP, according to \cref{prop:asymptoticerror}, we have
\eqn{\label{eq:gamma_ggp_asymp_error}
R_n  \sim \frac{(\kappa-\sigma)\Gamma(\kappa-\sigma)^{\frac{1}{\sigma}}}{(1-\sigma)\Gamma(\kappa)^{\frac{1}{\sigma}}} \bigg( \frac{\alpha}{\Gamma(1-\sigma)} \bigg) \sigma^{1-\frac{1}{\sigma}} n^{1-\frac{1}{\sigma}}.
}
\cref{fig:gamma_ggp_asymptotic} compares empirical approximate errors with $\hat n = 10^6$ to \eqref{eq:gamma_ggp_asymp_error} with different values of $\sigma$. We fixed $\kappa=1$ here. One can see that
the approximate error quickly approaches asymptotic errors.

\begin{figure}
\centering
\subfigure[SP-Gamma]{\includegraphics[width=0.325\linewidth]{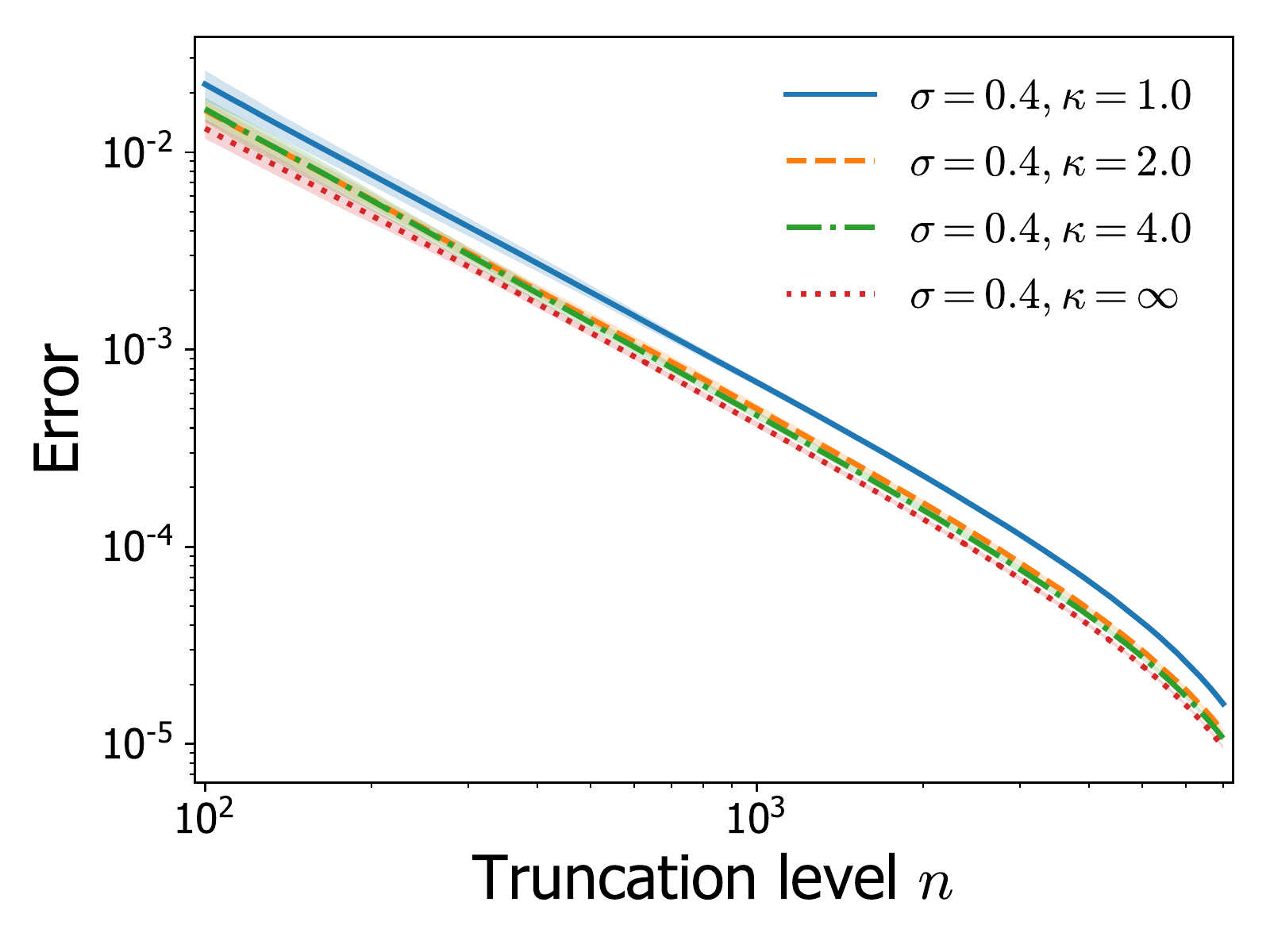}}
\subfigure[SP-iGamma]{\includegraphics[width=0.325\linewidth]{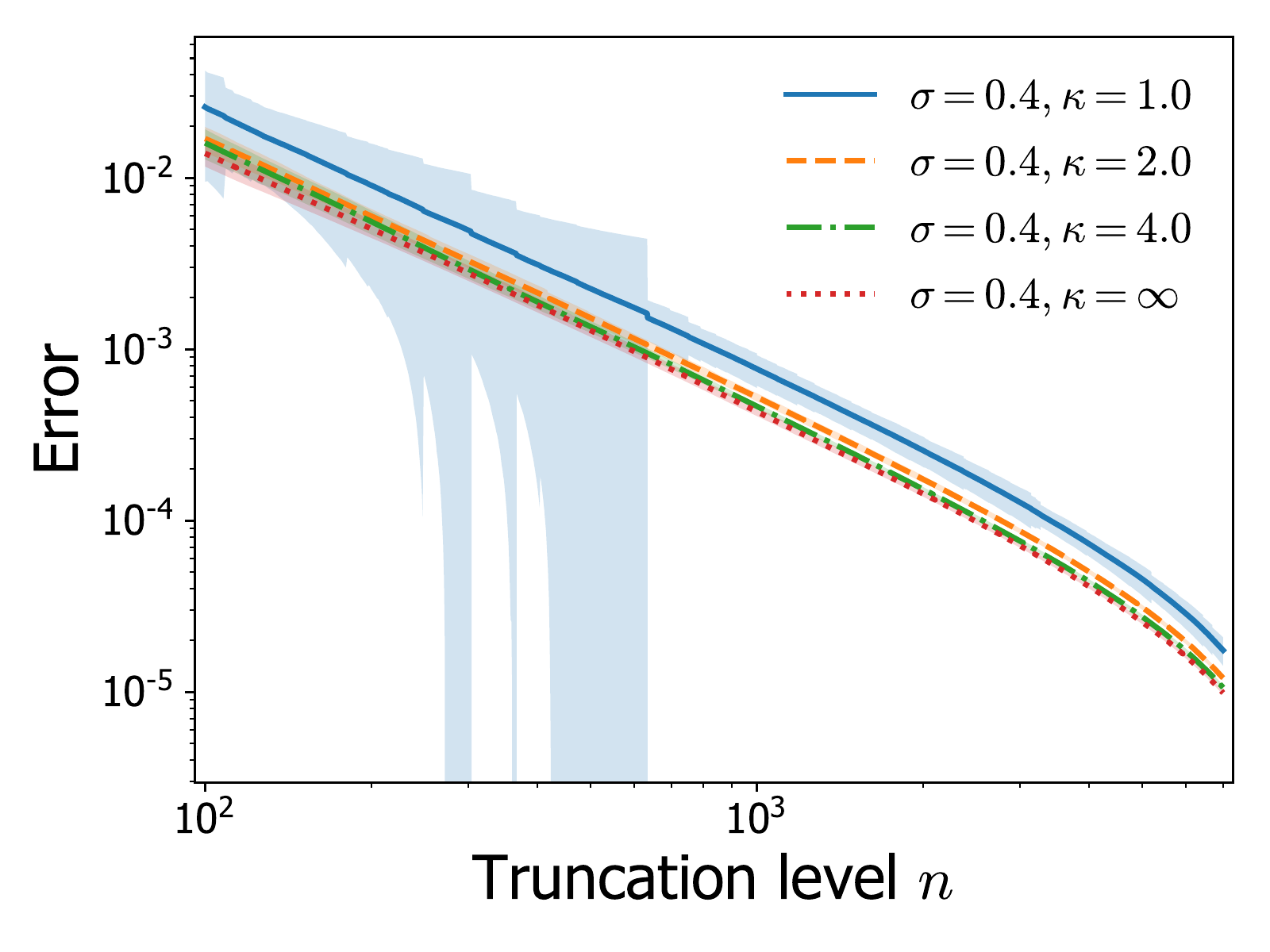}}
\subfigure[GGP-Gamma]{\includegraphics[width=0.325\linewidth]{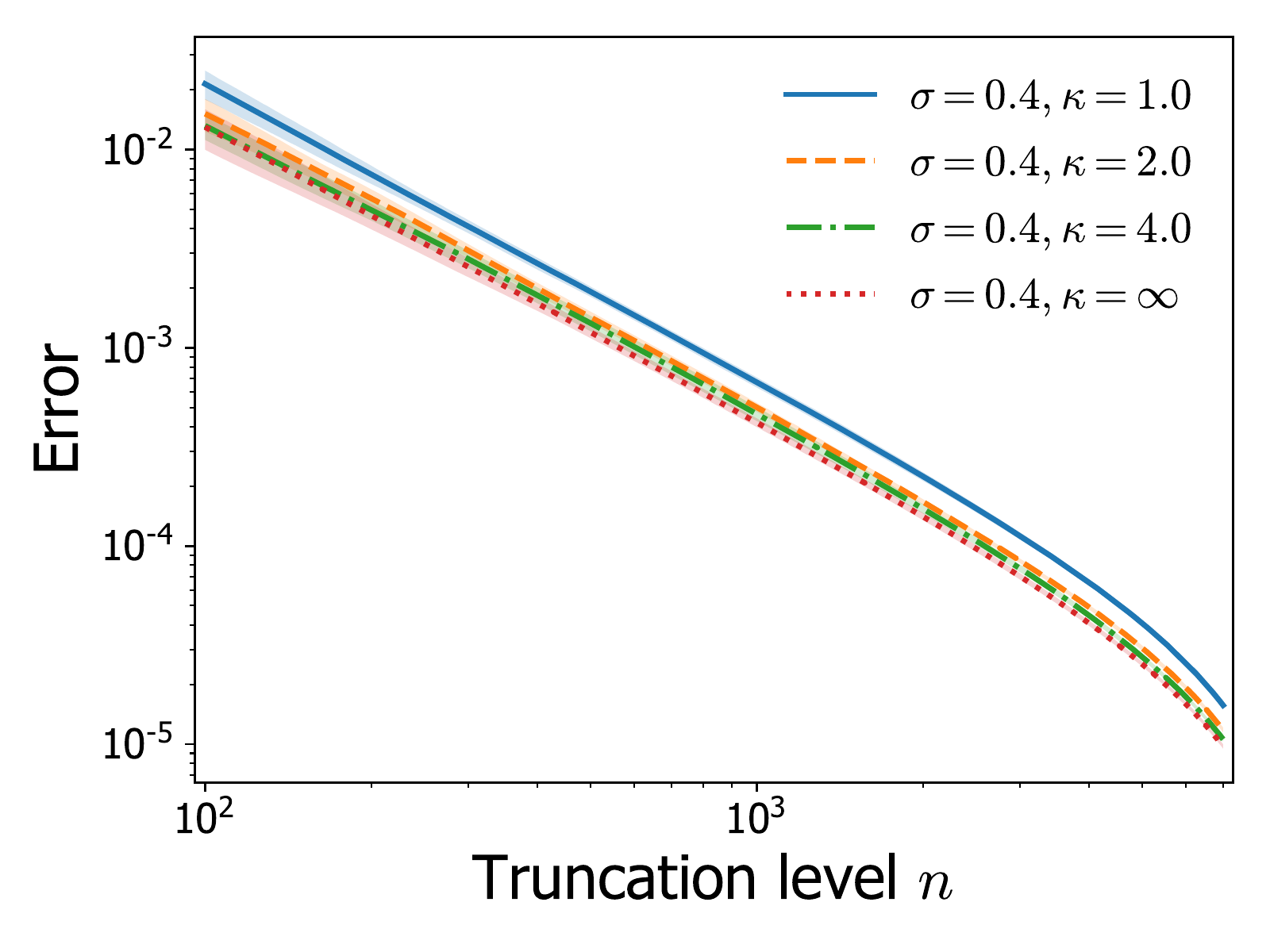}}
\caption{Gamma arrival times, inverse gamma arrival times for stable processes (a-b) and gamma arrival times for GGP (c). Plotted for whole range. See how
variances diminishes. Note also that in case of inverse gamma arrival times with $\kappa=1$, the variances diverge as our theory predicts.}
\label{fig:error_long_range}
\end{figure}

\begin{figure}
\centering
\subfigure[Gamma]{\includegraphics[width=0.325\linewidth]{figures/gamma_C1.pdf}}
\subfigure[iGamma]{\includegraphics[width=0.325\linewidth]{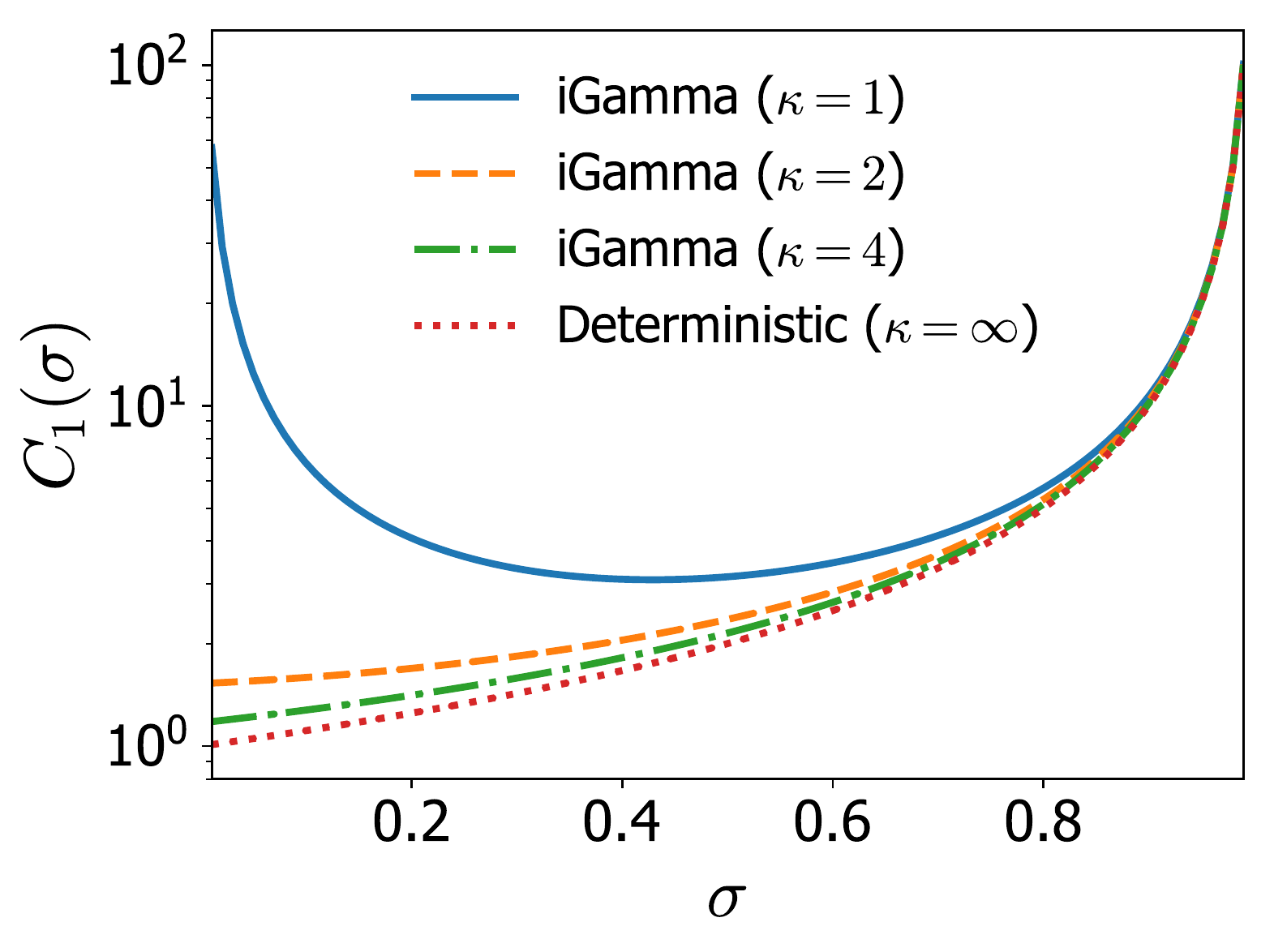}}
\subfigure[Gamma vs. iGamma]{\includegraphics[width=0.325\linewidth]{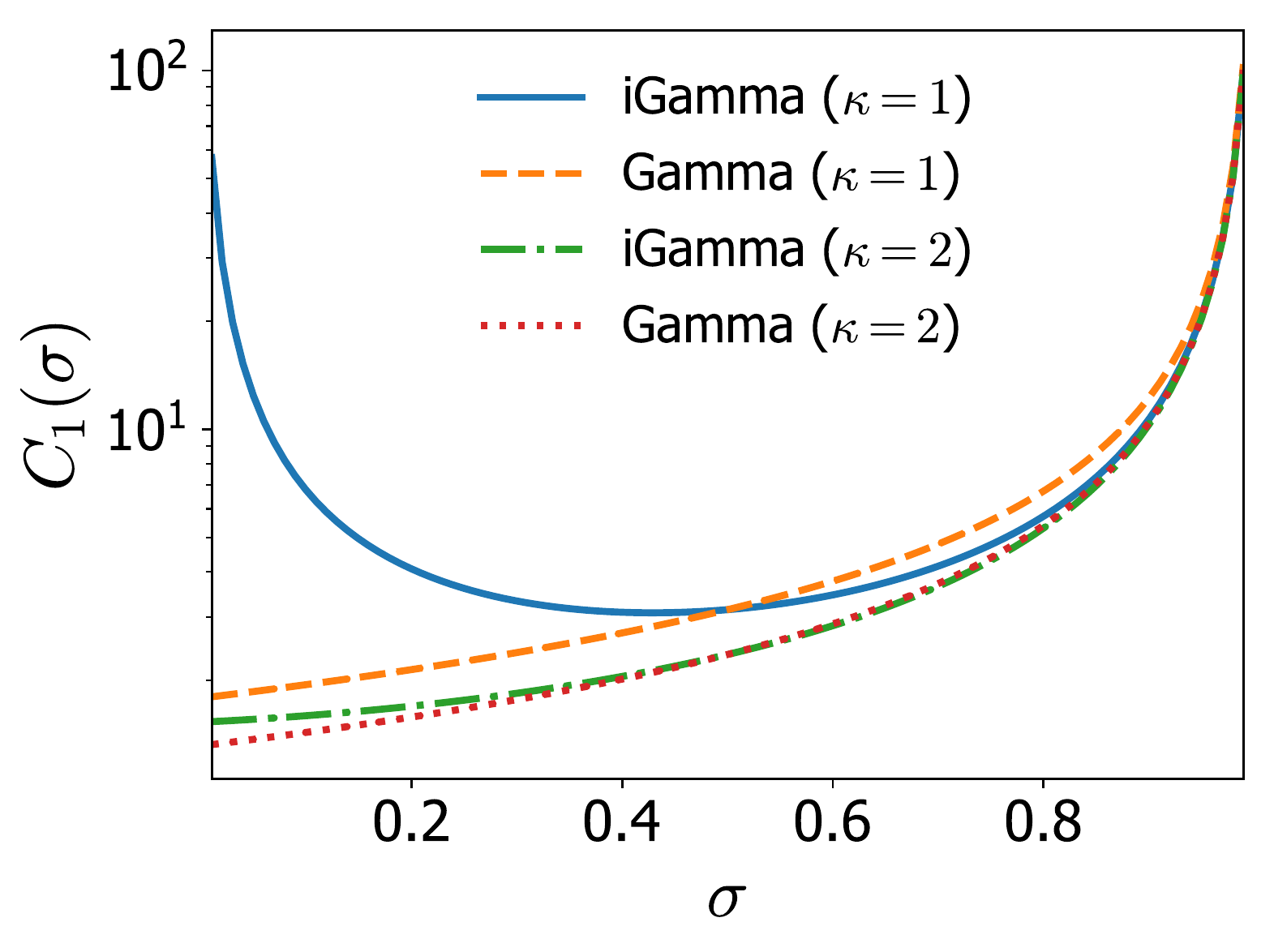}}
\caption{$C_1(\sigma)$ values for Gamma (a), inverse gamma (b) arrival times, and comparison between them (c).}
\label{fig:C1}
\end{figure}

\begin{figure}
\centering
\subfigure[$\sigma=0.4$]{\includegraphics[width=0.45\linewidth]{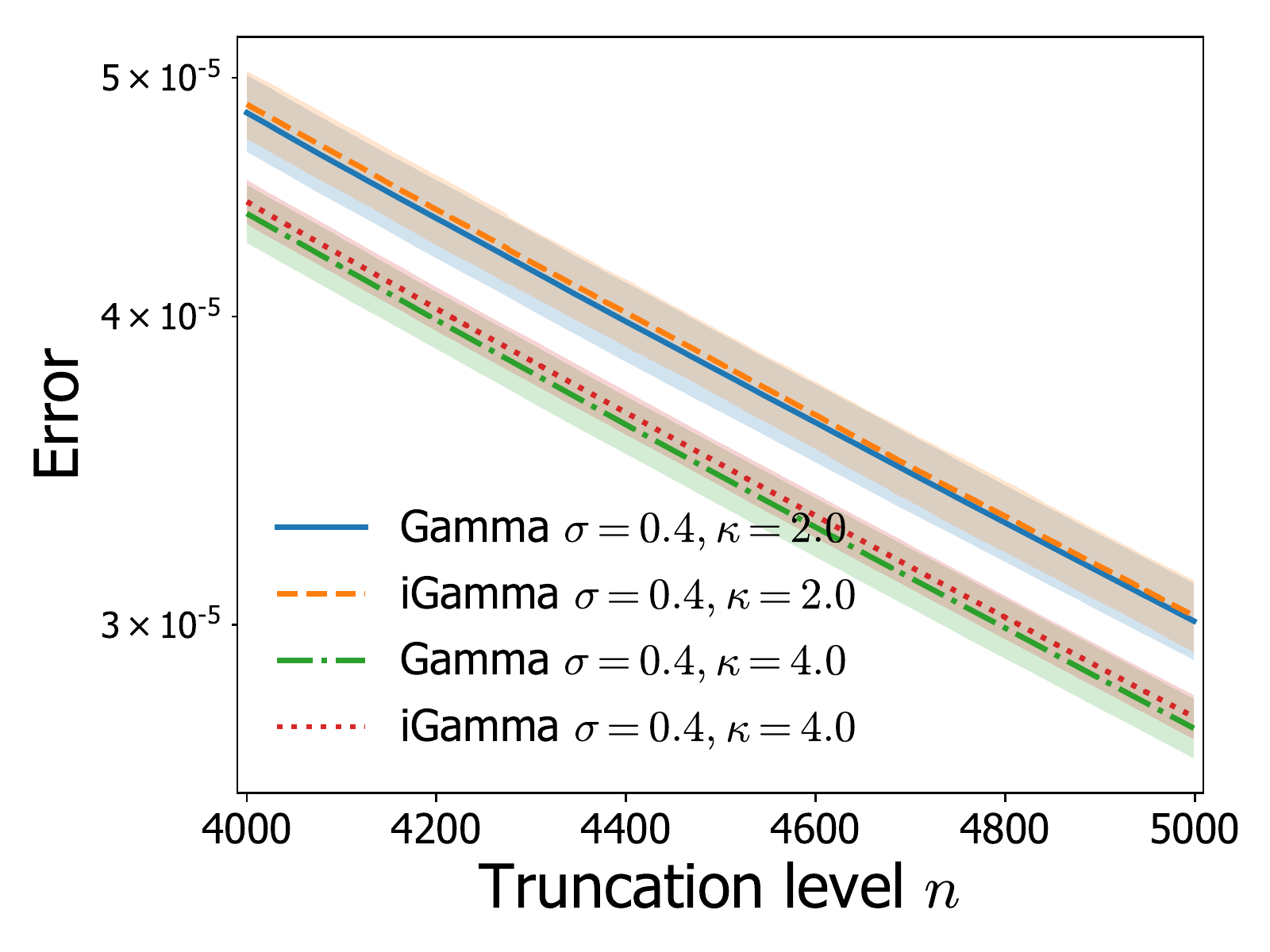}}
\subfigure[$\sigma=0.7$]{\includegraphics[width=0.45\linewidth]{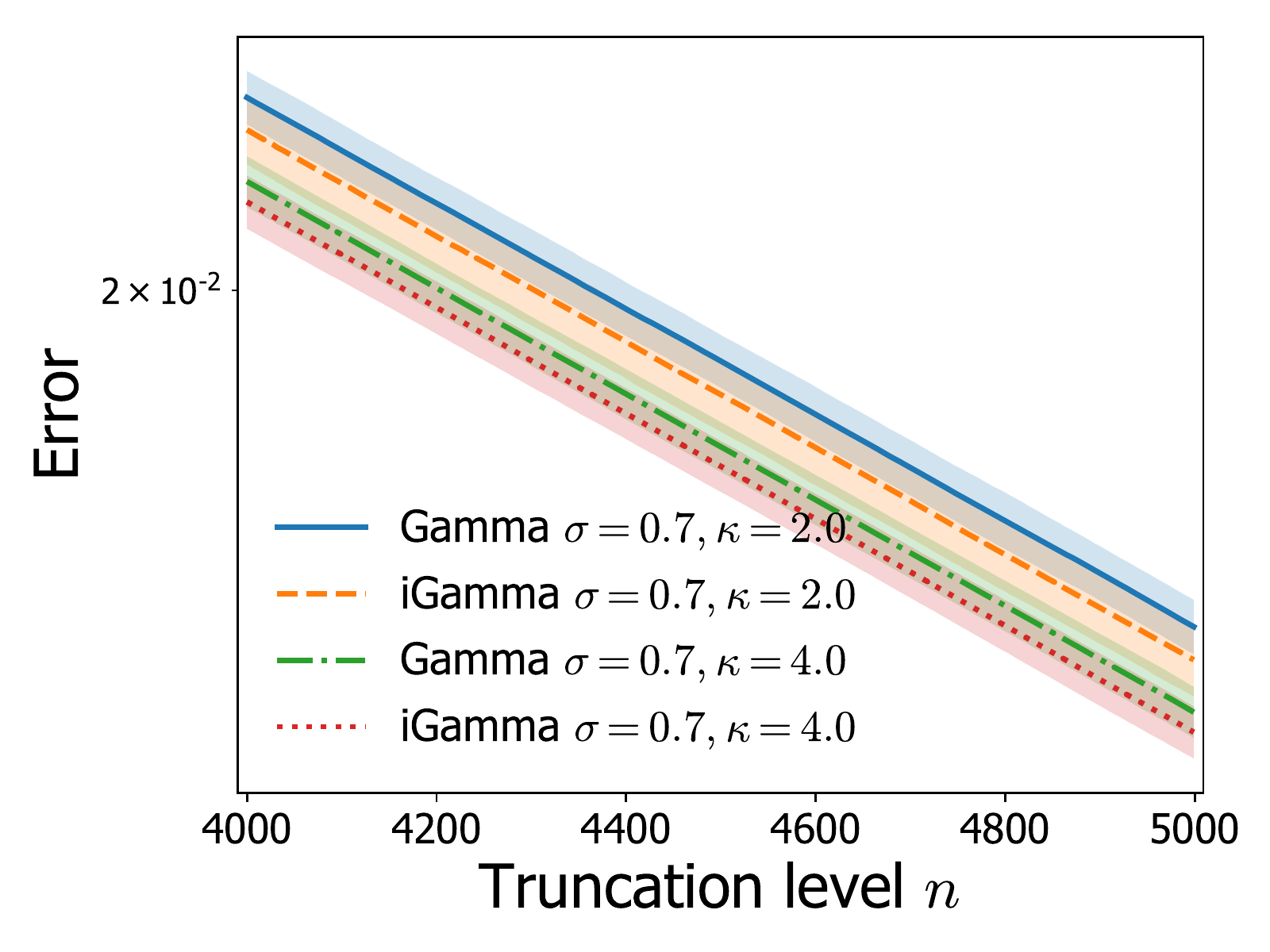}}
\caption{Gamma arrival times vs. inverse gamma arrival times for stable process, with $\sigma=0.4$ (a) and $\sigma=0.7$ (b). Gamma is better when $\sigma < 0.5$, and inverse gamma is better when $\sigma > 0.5$ as predicted in \cref{fig:C1} (c).}
\label{fig:gamma_vs_igamma}
\end{figure}

\begin{figure}
\centering
\subfigure[$\sigma = 0.4$]{\includegraphics[width=0.45\linewidth]{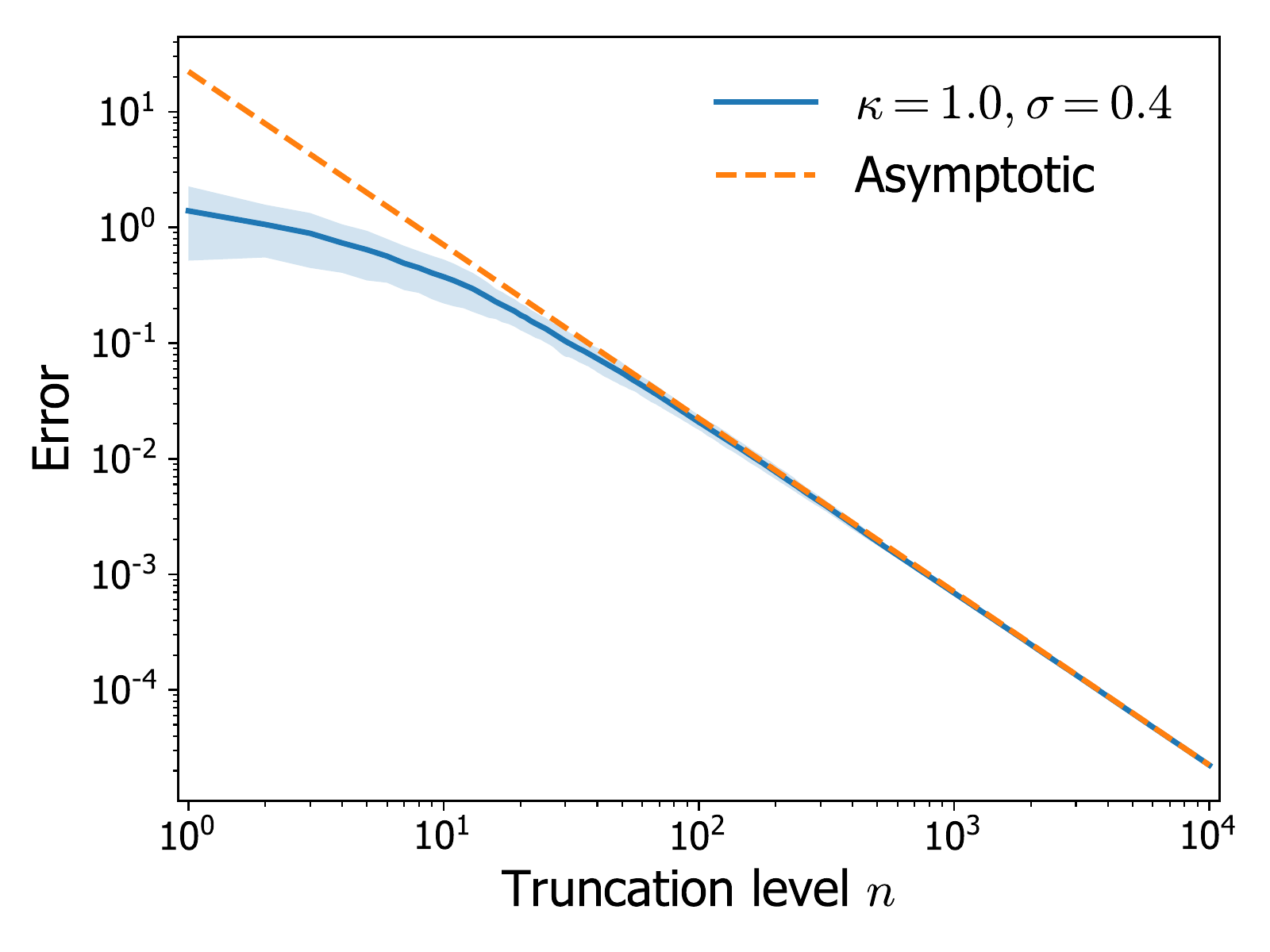}}
\subfigure[$\sigma = 0.7$]{\includegraphics[width=0.45\linewidth]{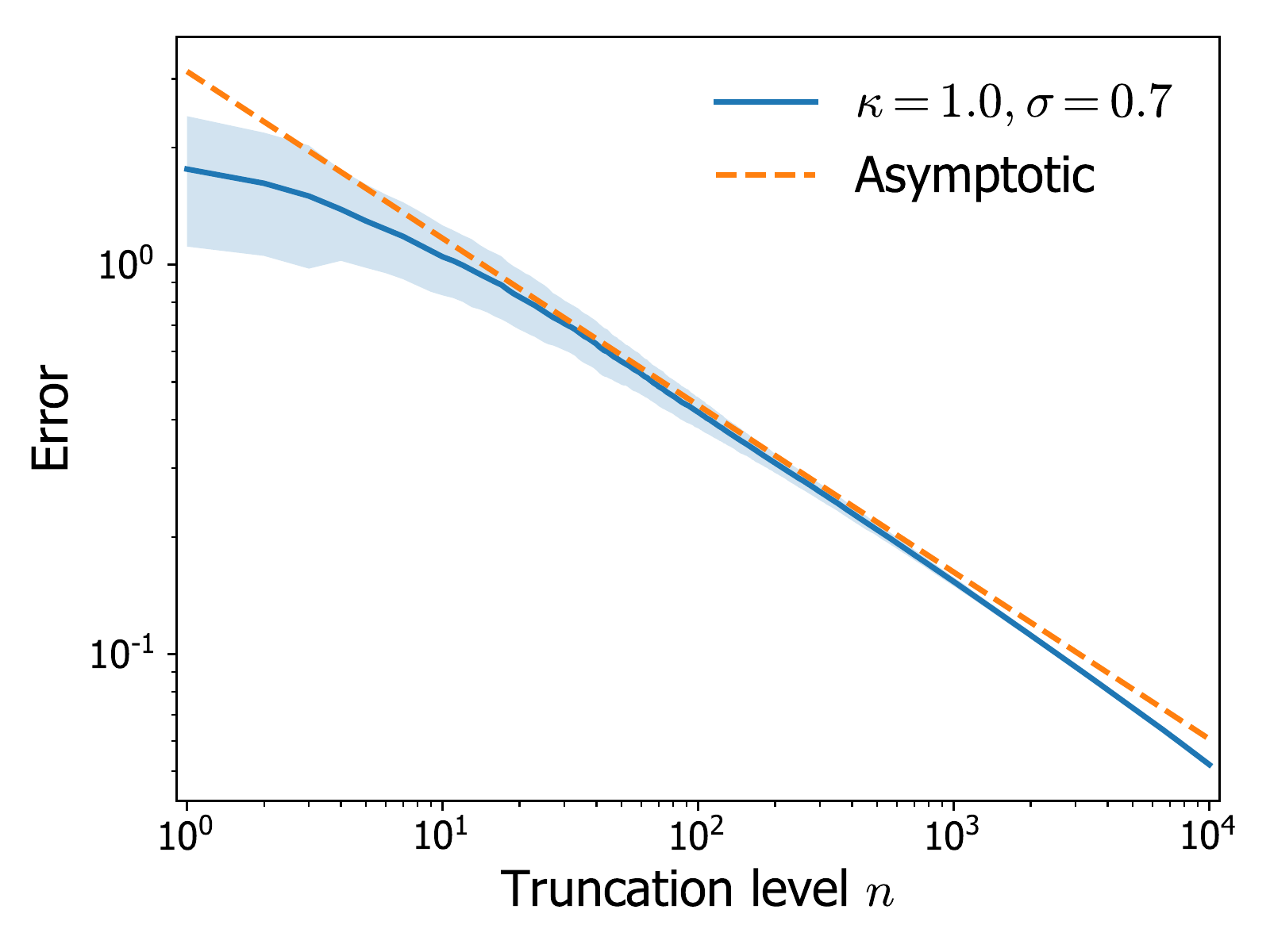}}
\caption{Empirical approximation error $R_{n,\hat n}$ compared to asymptotic error $R_n$, for (a) $\sigma=0.4$ and (b) $\sigma=0.7$.}
\label{fig:gamma_ggp_asymptotic}
\end{figure}

\section{Example on normalized GGP mixture models}
\label{sec:nggp_mixture}
Consider a hierarchical Bayesian model
\[
G = \sum_{i=1}^\infty W_i \delta_{\theta_i} \sim \crm(\rho, H), \,\,\,
\theta_j \iidsim \frac{G}{G(S)}, \,\,\,
X_j | \theta_j \iidsim L(\theta_j) \textrm{ for } j=1,\dots, m.
\]
We approximate infinite dimensional process $G$ with finite iid process $\tilde G_n$. Then, the rest of the model can be rewritten as
\[
Z_j | \widetilde G_n \iidsim \catdist\bigg(\frac{\widetilde w_{n,i}}{\sum_{i'=1}^n \widetilde w_{n,i'}}\bigg), \quad
X_j | z_j, \widetilde G_n \iidsim L(\theta_{z_j}),
\]
We construct $\widetilde G_n$ via gamma arrival times with $\kappa > 1$. Using \cref{prop:asymptoticPsiinv}, we have
\[
\widetilde\varphi_n(w) = \etgbfry\bigg( w ; \kappa, \sigma, \bigg( \frac{\sigma\Gamma(\kappa)\Gamma(1-\sigma)n}{\alpha\Gamma(\kappa-\sigma)}\bigg)^{\frac{1}{\sigma}}, \tau\bigg).
\]
Note that we used the function $f(n) = \big(\frac{\sigma \Gamma(\kappa)\Gamma(1-\sigma)n}{\alpha\kappa^\sigma\Gamma(\kappa-\sigma)}\big)^{\frac{1}{\sigma}}$ in place of $\Psi^{-1}(n)$,
thus both evaluation of the pdf and sampling can be done without any numerical approximation. The joint density of the mixture model is then written as
\[
w_\bullet^{-n}\prod_{k=1}^n w_k^{m_k} \widetilde\varphi_n(w_k) \bigg[ \prod_{j: z_j=k} \ell(x_j | \theta_k) \bigg] h(\theta_k),
\]
where $\ell$ and $h$ are the density for $L$ and $H$, $w_\bullet \defas \sum_{i=1}^n w_i$, and $m_k \defas \sum_{j=1}^m \1{z_j=k}$. Now we are free to any posterior inference algorithm, such as variational inference or stochastic gradient MCMC as in \cite{Lee2016}.
\end{appendices}

\end{document}